\documentclass[12pt]{amsart}
\usepackage[utf8]{inputenc}
\usepackage{amssymb}
\usepackage{hyperref}
\usepackage[final]{showkeys} %[final] is the option that removes them

\input xy
\xyoption{all}

\theoremstyle{definition}
\newtheorem{mydef}{Definition}[section]
\newtheorem{lem}[mydef]{Lemma}
\newtheorem{thm}[mydef]{Theorem}
\newtheorem{conjecture}[mydef]{Conjecture}
\newtheorem{cor}[mydef]{Corollary}
\newtheorem{claim}[mydef]{Claim}
\newtheorem{question}[mydef]{Question}
\newtheorem{hypothesis}[mydef]{Hypothesis}
\newtheorem{prop}[mydef]{Proposition}
\newtheorem{defin}[mydef]{Definition}
\newtheorem{example}[mydef]{Example}
\newtheorem{remark}[mydef]{Remark}
\newtheorem{notation}[mydef]{Notation}
\newtheorem{fact}[mydef]{Fact}

\newcommand{\fct}[2]{{}^{#1}#2}

\newcommand{\footnotei}[1]{}

\newcommand{\ba}{\bar{a}}
\newcommand{\bb}{\bar{b}}

\newcommand{\bigN}{\widehat{N}}
\newcommand{\bigf}{\widehat{f}}

\newcommand{\Ksatpp}[2]{{#1}^{#2\text{-sat}}}
\newcommand{\Ksatp}[1]{\Ksatpp{K}{#1}}

\renewcommand{\bigg}{\widehat{g}}
\newcommand{\dom}[1]{\text{dom}(#1)}
\newcommand{\ran}[1]{\text{ran}(#1)}

\newcommand{\cf}[1]{\text{cf} (#1)}
\newcommand{\seq}[1]{\langle #1 \rangle}
\newcommand{\rest}{\upharpoonright}

\newcommand{\s}{\mathfrak{s}}

\newcommand{\is}{\mathfrak{i}}

\def\lta{<}
\def\lea{\le}

\def\gea{\ge}

\def\lee{\preceq}

\def\ltu{\lta_{\text{univ}}}
\def\leu{\lea_{\text{univ}}}

\newcommand{\K}{K}

\newbox\noforkbox \newdimen\forklinewidth
\forklinewidth=0.3pt \setbox0\hbox{$\textstyle\smile$}
\setbox1\hbox to \wd0{\hfil\vrule width \forklinewidth depth-2pt
 height 10pt \hfil}
\wd1=0 cm \setbox\noforkbox\hbox{\lower 2pt\box1\lower
2pt\box0\relax}
\def\unionstick{\mathop{\copy\noforkbox}\limits}
\newcommand{\nf}{\unionstick}
\newcommand{\nfs}[4]{#2 \nf_{#1}^{#4} #3}

\def\1nf{\unionstick^{(1)}}

\def\2nf{\unionstick^{(2)}}
\def\3nf{\unionstick^{(3)}}

\newcommand{\tp}{\text{tp}}
\newcommand{\gtp}{\text{gtp}}

\newcommand{\gS}{\text{gS}}

\newcommand{\hanf}[1]{h (#1)}

\newcommand{\Axfr}{\text{AxFri}_1}

\newcommand{\goodp}{\text{good}^+}
\newcommand{\Eat}{E_{\text{at}}}
\newcommand{\cl}{\text{cl}}
\newcommand{\pre}{\text{pre}}
\newcommand{\F}{\mathcal{F}}
\newcommand{\LS}{\text{LS}}

\title[Categoricity in universal classes: part I]{Shelah's eventual categoricity conjecture in universal classes: part I}
\date{\today}

\author{Sebastien Vasey}
\email{sebv@cmu.edu}
\urladdr{http://math.cmu.edu/\textasciitilde svasey/}
\address{Department of Mathematical Sciences, Carnegie Mellon University, Pittsburgh, Pennsylvania, USA}
\thanks{This material is based upon work done while the author was supported by the Swiss National Science Foundation under Grant No.\ 155136.}

\begin{document}

\begin{abstract}
  We prove:

  \begin{thm}\label{abstract-thm-0}
  Let $K$ be a universal class. If $K$ is categorical in cardinals of arbitrarily high cofinality, then $\K$ is categorical on a tail of cardinals.
\end{thm}

The proof stems from ideas of Adi Jarden and Will Boney, and also relies on a deep result of Shelah. As opposed to previous works, the argument is in ZFC and does not use the assumption of categoricity in a successor cardinal. The argument generalizes to abstract elementary classes (AECs) that satisfy a locality property and where certain prime models exist. Moreover assuming amalgamation we can give an explicit bound on the Hanf number and get rid of the cofinality restrictions:

\begin{thm}\label{prime-caract-0}
  Let $K$ be an AEC with amalgamation. Assume that $K$ is fully $\LS (K)$-tame and short and has primes over sets of the form $M \cup \{a\}$. Write $H_2 := \beth_{\left(2^{\beth_{\left(2^{\LS (\K)}\right)^+}}\right)^+}$. If $K$ is categorical in a $\lambda > H_2$, then $K$ is categorical in all $\lambda' \ge H_2$.
\end{thm}

\end{abstract}

\maketitle

\tableofcontents

\section{Introduction}

Morley's categoricity theorem \cite{morley-cip} states that a first-order countable theory that is categorical in some uncountable cardinal must be categorical in all uncountable cardinals. The result motivated much of the development of first-order classification theory (it was later generalized by Shelah \cite{sh31} to uncountable theories). 

Toward developing a classification theory for non-elementary classes, one can ask whether there is such a result for infinitary logics, e.g. for an $L_{\omega_1, \omega}$ sentence. In 1971, Keisler proved \cite[Section 23]{kei71} a generalization of Morley's theorem to this framework assuming in addition that the model in the categoricity cardinal is sequentially homogeneous. Unfortunately Shelah later observed using an example of Marcus \cite{marcus-counterexample} that Keisler's assumption does not follow from categoricity. Still in the late seventies Shelah proposed the following far-reaching conjecture:

\begin{conjecture}[Open problem D.(3a) in \cite{shelahfobook}]\label{categ-conj}
  If $L$ is a countable language and $\psi \in L_{\omega_1, \omega}$ is categorical in one $\lambda \ge \beth_{\omega_1}$, then it is categorical in all $\lambda' \ge \beth_{\omega_1}$.
\end{conjecture}

This has now become the central test problem in classification theory for non-elementary classes. Shelah alone has more than 2000 pages of approximations (for example \cite{sh48, sh87a, sh87b, makkaishelah, sh394, sh576, shelahaecbook, shelahaecbook2}). Shelah's results led him to introduce a semantic framework encompassing several different infinitary logics and algebraic classes \cite{sh88}: abstract elementary classes (AECs). In this framework, we can state an eventual version of the conjecture\footnote{The statement here appears in \cite[Conjecture N.4.2]{shelahaecbook}.}:

\begin{conjecture}[Shelah's eventual categoricity conjecture for AECs]
  An AEC that is categorical in a high-enough cardinal is categorical on a tail of cardinals.
\end{conjecture}
\begin{remark}\label{high-enough-rmk}
  A more precise statement is that there should be a function $\mu \mapsto \lambda_\mu$ such that every AEC $K$ categorical in some $\lambda \ge \lambda_{\LS (K)}$ is categorical in all $\lambda' \ge \lambda_{\LS (K)}$. By a similar argument as for the existence of Hanf numbers \cite{hanf-number} (see \cite[Conclusion 15.13]{baldwinbook09}), Shelah's eventual categoricity conjecture for AECs is equivalent to the statement that an AEC categorical in \emph{unboundedly many} cardinals is categorical on a tail of cardinals. We will use this equivalence freely. Note that Theorem \ref{prime-caract-0} gives an \emph{explicit}\footnote{We thank John Baldwin for helpful conversation on the topic.} bound for $\lambda_\mu$, so proves a stronger statement than just Shelah's eventual categoricity conjecture for universal classes with amalgamation\footnote{We are not sure how to make the distinction precise. Maybe one can call the \emph{computable} eventual categoricity conjecture the statement that has the additional requirement that $\mu \mapsto \lambda_\mu$ be computable, where computable can be defined as in \cite{blsh992}. Note that in Shelah's original categoricity conjecture, $\lambda_{\mu}$ is $\beth_{(2^{\mu})^+}$, see \cite[6.14.(3)]{sh702}.}.
\end{remark}

Positive results are known in less general frameworks: For homogeneous model theory by Lessmann \cite{lessmann2000} and more generally for $\aleph_0$-tame\footnote{Tameness is a locality property for orbital types introduced by Grossberg and VanDieren in \cite{tamenessone}.} simple finitary AECs by Hyttinen and Kesälä \cite{categ-finit} (note that these results apply only to countable languages). In uncountable languages, Grossberg and VanDieren proved the conjecture in tame AECs categorical in a successor cardinal \cite{tamenesstwo, tamenessthree}. Later Will Boney pointed out that tameness follows\footnote{Recently Boney and Unger \cite{lc-tame-v3-toappear} established that the statement ``all AECs are tame'' is in fact \emph{equivalent} to a large cardinal axioms (the existence of a proper class of almost strongly compact cardinals). This result does not however say anything on the consistency strength of Shelah's eventual categoricity conjecture.} from large cardinals \cite{tamelc-jsl}, a result that (as pointed out in \cite{ct-accessible-v4-toappear}) can also be derived from a 25 year old theorem of Makkai and Paré (\cite[Theorem 5.5.1]{makkai-pare}). A combination of this gives that statements much stronger than Shelah's categoricity conjecture for a successor hold if there exists a proper class of strongly compact cardinals.

The question of whether categoricity in a sufficiently high \emph{limit} cardinal implies categoricity on a tail remains open (even in tame AECs). The central tool there is the notion of a \emph{good $\lambda$-frame}, a local axiomatization of forking which is the main concept in \cite{shelahaecbook}. After developing the theory of good $\lambda$-frames over several hundreds of pages, Shelah claims to be able to prove the following (see \cite[Discussion III.12.40]{shelahaecbook}, a proof should appear in \cite{sh842}):

\begin{claim}\label{shelah-xxx}
Assume that $2^{\theta} < 2^{\theta^+}$ for all cardinals $\theta$. Let $K$ be an AEC such that there is an $\omega$-successful\footnote{See Appendix \ref{good-frame-appendix} for a definition of good frames and the related technical terms.} $\goodp$ $\lambda$-frame with underlying class $K_\lambda$. If $K$ is categorical in $\lambda$ and in some $\mu > \lambda^{+\omega}$, then $\K$ is categorical in \emph{all} $\mu > \lambda^{+\omega}$. 
\end{claim}

Assuming amalgamation and Claim \ref{shelah-xxx}, Shelah obtains the eventual categoricity conjecture \cite[Theorem IV.7.12]{shelahaecbook} (or see \cite[Section 11]{downward-categ-tame-apal} for an exposition):

\begin{fact}\label{ap-categ}
  Assume Claim \ref{shelah-xxx} and $2^{\theta} < 2^{\theta^+}$ for all cardinals $\theta$. Then an AEC with amalgamation categorical in \emph{some} $\lambda \ge \hanf{\aleph_{\LS (\K)^+}}$ is categorical in \emph{all} $\lambda' \ge \hanf{\aleph_{\LS (\K)^+}}$.
\end{fact} 

Here and throughout the rest of this paper, we use the notation from \cite[Chapter 14]{baldwinbook09}:

\begin{notation}\label{hanf-notation}
  For $\theta$ an infinite cardinal, let $\hanf{\theta} := \beth_{\left(2^{\theta}\right)^+}$. For a fixed AEC $\K$, write $H_1 := \hanf{\LS (\K)}$, $H_2 := \hanf{H_1}$.
\end{notation}

Note that Fact \ref{ap-categ} applies in particular to homogeneous model theory and finitary AECs with uncountable language (the latter case could not previously be dealt with).

Now a conjecture of Grossberg made in 1986 (see Grossberg \cite[Conjecture 2.3]{grossberg2002}) is that categoricity of an AEC in a high-enough cardinal should imply amalgamation (above a certain Hanf number). This is especially relevant considering that all the positive results above assume amalgamation. In the presence of large cardinals, Grossberg's conjecture is known to be true (This was first pointed out by Will Boney for general AECs, see \cite[Theorem 4.3]{tamelc-jsl} and the discussion around Theorem 7.6 there. The key is that the proofs in \cite[Proposition 1.13]{makkaishelah} or the stronger \cite{kosh362} which are for classes of models of an $L_{\kappa, \omega}$ sentence, $\kappa$ a large cardinal, carry over to AECs $K$ with $\LS (K) < \kappa$). In recent years it has been shown that several results that could be proven using large cardinals can be proven using just the model-theoretic assumption of tameness or shortness (see all of the above papers on tameness and for example \cite{sv-infinitary-stability-afml,bv-sat-v3}). Thus one can ask whether tameness suffices to get amalgamation from categoricity. In general, this is not known. The only approximation is a result of Adi Jarden \cite{jarden-tameness-apal} discussed more at length in Section \ref{ap-sec}. Our contribution is a weak version of amalgamation which one can assume alongside tameness to prove Grossberg's conjecture:

\textbf{Corollary \ref{tame-ap-0}.}
  Let $K$ be tame AEC categorical in unboundedly many cardinals. If $K$ is eventually syntactically characterizable\footnote{A technical condition discussed more at length in Section \ref{ap-sec}.} and has weak amalgamation (see Definition \ref{weak-ap-def}), then there exists $\lambda$ such that $K_{\ge \lambda}$ has amalgamation. 

The proof uses a deep result of Shelah showing that a categorical AEC is well-behaved in a specific cardinal, then uses tameness and weak amalgamation to transfer the good behavior up.

We apply our result to \emph{universal classes}. Universal classes were introduced by Shelah in \cite{sh300-orig} as an important framework where he thought finding dividing lines should be possible\footnote{We were told by Rami Grossberg that another motivation was to study certain non first-order classes of modules.}. For many years, Shelah has claimed a main gap theorem for these classes but the full proof has not appeared in print. The most recent version is Chapter V of \cite{shelahaecbook2} which contains hundreds of pages of approximations. The methods used are stability theory inside a model (averages) as well as combinatorial tools to build many models. Here we show that universal classes are tame\footnote{This uses an argument of Will Boney.} (in fact fully $(<\aleph_0)$-tame and short) and have weak amalgamation. Moreover Shelah has shown\footnote{In fact, Shelah asserts that the cofinality restriction is not necessary but Will Boney and the author have found a gap in Shelah's argument, and Shelah's fix has not yet been published. See the beginning of Section \ref{ap-sec}.} that categoricity in cardinals of arbitrarily high cofinality implies that the class is eventually syntactically characterizable. Thus combining Corollary \ref{tame-ap-0} and Fact \ref{ap-categ} we can already prove Theorem \ref{abstract-thm-0} assuming the weak generalized continuum hypothesis and Claim \ref{shelah-xxx}. If the universal class is categorical in unboundedly many successor cardinals, we can use \cite{tamenessthree} instead to get a categoricity transfer in ZFC.

By relying on Shelah's analysis of frames in \cite[Chapter III]{shelahaecbook} as well as the frame transfer theorems in \cite{ext-frame-jml, tame-frames-revisited-v6-toappear}, we can also prove that Claim \ref{shelah-xxx} holds in ZFC for universal classes (this uses the proof of Corollary \ref{tame-ap-0}). We deduce Theorem \ref{abstract-thm-0} in the abstract (see Corollary \ref{abstract-thm-0-proof}). Note that the result also holds in uncountable languages.

Our results apply to a more general context than universal classes: fully tame and short AECs with amalgamation which have a prime model over every set of the form $M \cup \{a\}$ for $M$ a model (this is Theorem \ref{prime-caract-0} in the abstract, see Theorem \ref{categ-nice-ap} for a proof). Note that existence of prime models over sets of the form $M \cup \{a\}$ already played a crucial role in the proof of the categoricity transfer theorem for excellent classes of models of an $L_{\omega_1, \omega}$ sentences \cite[Theorem 5.9]{sh87b} (in fact, our proof works also in this setting). Since in that case we are assuming amalgamation, there is no cofinality restrictions and the Hanf number can be explicitly computed.

Theorem \ref{prime-caract-0} shows that, at least assuming amalgamation, tameness and shortness, the existence of primes is the \emph{only} obstacle. Since amalgamation and full tameness and shortness follow from large cardinals \cite{tamelc-jsl}, we obtain:

\begin{thm}\label{strongly-compact-thm}
  Let $K$ be an AEC and let $\kappa > \LS (K)$ be strongly compact. Assume that in $K_{\ge \kappa}$ there are prime models over sets of the form $M \cup \{a\}$. If $K$ is categorical in a $\lambda > \hanf{\hanf{\kappa}}$, then $K$ is categorical in all $\lambda' \ge \hanf{\hanf{\kappa}}$.
\end{thm}

When $\K$ is a universal class, we can replace the strongly compact with a measurable (Theorem \ref{measurable-thm}). 

There remains one question: can the conclusion of Theorem \ref{abstract-thm-0} be obtained from only categoricity in a single cardinal (without cofinality restriction)? We answer positively in a sequel\footnote{Since the initial circulation of this paper (in June 2015), there have been several other improvements: Question \ref{tame-shortness-q} has been answered positively \cite{categ-primes-v6-toappear} and the categoricity threshold of Theorem \ref{prime-caract-0} has been improved from $H_2$ to $H_1$ \cite{downward-categ-tame-apal}. All these improvements rely on the results of this paper.} \cite{categ-universal-2-v3-toappear}. Here, let us note that Theorem \ref{abstract-thm-0} generalizes to fully tame and short AECs with primes, but universal classes have better properties (as demonstrated by Shelah in \cite[Chapter V]{shelahaecbook2}), so there is still room for improvement\footnote{In fact, we had claimed in an earlier version of this work to be able to prove the full categoricity conjecture for universal classes but our argument contained an error.}.

The paper is organized as follows. In Section \ref{inter-sec}, we recall the definition of universal classes and more generally of AECs which admit intersections (a notion introduced by Baldwin and Shelah in \cite{non-locality}), give examples, and prove some basic properties. In Section \ref{tameness-sec}, we prove that universal classes are fully $(<\aleph_0)$-tame and short. In Section \ref{ap-sec} we give conditions under which amalgamation follows from categoricity (in more general classes than universal classes). In Section \ref{categ-transfer-sec}, we prove a categoricity transfer in universal classes that have amalgamation and more generally in fully tame and short AECs with primes and amalgamation.

To avoid cluttering the paper, we have written the technical definitions and results on independence needed for the paper (but not crucial to a conceptual understanding) in Appendix \ref{good-frame-appendix}. In Appendix \ref{proof-appendix}, we prove Fact \ref{not-unidim-frame}, a result of Shelah which is crucial to our argument but whose proof is only implicit in Shelah's book. In Appendix \ref{indep-sec}, we give some properties of independence in AECs which admit intersections.

A word on the background needed to read this paper: we assume familiarity with a basic text on AECs such as \cite{baldwinbook09} or \cite{grossbergbook} and refer the reader to the preliminaries of \cite{sv-infinitary-stability-afml} for more details and motivations on the notation and concepts used here. Familiarity with good frames \cite[Chapter II]{shelahaecbook} would be helpful, although the basics are reviewed in Appendix \ref{good-frame-appendix}. The proof of the two theorems in the abstract relies on the construction of a good frame in \cite{ss-tame-jsl}, and more generally\footnote{Although since this paper has first been made public, an improvement has been published that avoids dealing with global independence relation, see \cite{categ-primes-v6-toappear} (we have kept the original proof to avoid changing history and also because Appendix \ref{good-frame-appendix} is useful in other contexts).} on the study of global independence relations in \cite{indep-aec-apal}. Some material from Chapter III of \cite{shelahaecbook} is implicitly used there. To get amalgamation and prove Theorem \ref{abstract-thm-0}, the hard arguments of \cite[Chapter IV]{shelahaecbook} are used. However we do not rely on them once amalgamation has been obtained (so for example Theorem \ref{prime-caract-0} does \emph{not} rely on \cite[Chapter IV]{shelahaecbook}). Finally, let us note that a more leisurely overview of the proof of Theorem \ref{abstract-thm-0} will appear in \cite{bv-survey-v4-toappear}. We have also written a short outline of the proof in \cite{lazy-univ-v2}.

This paper was written while working on a Ph.D.\ thesis under the direction of Rami Grossberg at Carnegie Mellon University and I would like to thank Professor Grossberg for his guidance and assistance in my research in general and in this work specifically. I thank Will Boney for pointing me to AECs which admit intersections, for helpful conversations, and for good feedback. I thank John Baldwin, Adi Jarden, and the referee for useful feedback that greatly helped me improve the presentation of this paper. 

\section{AECs which admit intersections}\label{inter-sec}

Recall:

\begin{defin}[\cite{sh300-orig}]
  A class of structures $K$ is \emph{universal} if:
  
  \begin{enumerate}
    \item It is a class of $L$-structures for a fixed language $L = L (K)$, closed under isomorphisms.
    \item If $\seq{M_i : i < \delta}$ is $\subseteq$-increasing in $K$, then $\bigcup_{i < \delta} M_i \in K$.
    \item If $M \in K$ and $M_0 \subseteq M$, then $M_0 \in K$.
  \end{enumerate}
\end{defin}
\begin{example} \
\begin{enumerate}
  \item The class of models of a universal $L_{\lambda, \omega}$ theory (that is, a set of sentences of the form $\forall x_0 \ldots \forall x_{n - 1} \psi$, with $\psi$ a quantifier-free $L_{\lambda, \omega}$-formula) is universal.
  \item Not all elementary classes are universal but some universal classes are not elementary (locally finite groups are one example).
  \item Coloring classes \cite{coloring-classes-jsl} are universal classes. This shows that the behavior of amalgamation is non-trivial even in universal classes: some coloring classes can have amalgamation up to $\beth_\alpha$ for some $\alpha < \LS (K)^+$ and fail to have it above $\beth_{\LS (\K)^+}$. Other universal classes with non-trivial amalgamation spectrum appear in \cite{locally-finite-aec-toappear}.
  \item If $K$ is a universal class in a countable vocabulary with:
    \begin{enumerate}
    \item Arbitrarily large models.
    \item Joint embedding.
    \item Disjoint amalgamation (see Definition \ref{ap-def}).
    \end{enumerate}

    Then $K$ is a finitary abstract elementary class in the sense of Hyttinen and Kesälä \cite[Definition 2.9]{finitary-aec}. We do not know whether $K$ is always simple (in the sense of \cite[Definition 6.1]{categ-finit}. In any case, the arguments of Hyttinen and Kesälä only deal with countable languages.
\end{enumerate}
\end{example}

Universal classes are abstract elementary classes: 

\begin{defin}[Definition 1.2 in \cite{sh88}]\label{aec-def}
  An \emph{abstract elementary class} (AEC for short) is a pair $(K, \lea)$, where:

  \begin{enumerate}
    \item $K$ is a class of $L$-structured, for some fixed language $L = L(K)$. 
    \item $\lea$ is a partial order (that is, a reflexive and transitive relation) on $K$.
    \item $(K, \lea)$ respects isomorphisms: If $M \lea N$ are in $K$ and $f: N \cong N'$, then $f[M] \lea N'$. 
    \item If $M \lea N$, then $M \subseteq N$. 
    \item Coherence: If $M_0, M_1, M_2 \in K$ satisfy $M_0 \lea M_2$, $M_1 \lea M_2$, and $M_0 \subseteq M_1$, then $M_0 \lea M_1$;
    \item Tarski-Vaught axioms: Suppose $\delta$ is a limit ordinal and $\seq{M_i \in K : i < \delta}$ is an increasing chain. Then:

        \begin{enumerate}

            \item $M_\delta := \bigcup_{i < \delta} M_i \in K$ and $M_0 \lea M_\delta$.
            \item\label{smoothness-axiom} If there is some $N \in K$ so that for all $i < \delta$ we have $M_i \lea N$, then we also have $M_\delta \lea N$.

        \end{enumerate}

    \item Löwenheim-Skolem-Tarski axiom: There exists a cardinal $\lambda \ge |L(K)| + \aleph_0$ such that for any $M \in K$ and $A \subseteq |M|$, there is some $M_0 \lea M$ such that $A \subseteq |M_0|$ and $\|M_0\| \le |A| + \lambda$. We write $\LS (K)$ for the minimal such cardinal.
\end{enumerate}

  We often will not distinguish between the class $K$ and the pair $(K, \lea)$.
\end{defin}

\begin{remark}
  If $K$ is a universal class, then $(K, \subseteq)$ is an AEC with $\LS (K) = |L (K)| + \aleph_0$. We will use this fact freely. Note that $K$ may have finite models, and it is the case in several examples, see \cite{locally-finite-aec-toappear}.
\end{remark}

We now recall the definition of AECs that admit intersections, a notion introduced by Baldwin and Shelah. It is interesting to note that Baldwin and Shelah thought of admitting intersections as a weak version of amalgamation (see the conclusion of \cite{non-locality}).

\begin{defin}[Definition 1.2 in \cite{non-locality}]
  Let $K$ be an AEC.

  \begin{enumerate}
    \item Let $N \in K$ and let $A \subseteq |N|$. $N$ \emph{admits intersections over $A$} if there is $M_0 \lea N$ such that $|M_0| = \bigcap \{M \lea N \mid A \subseteq |M|\}$. $N$ \emph{admits intersections} if it admits intersections over all $A \subseteq |N|$.
    \item $K$ \emph{admits intersections} if $N$ admits intersections for all $N \in K$.
  \end{enumerate}
\end{defin}

\begin{example} \
  \begin{enumerate}
    \item If $K$ is a universal class, then $K$ admits intersections (and see Remark \ref{univ-class-closure}).
    \item If $K$ is a class of models of a first-order theory, then when $(K, \subseteq)$ admits intersections has been characterized by Rabin \cite{rabin-inter}.
    \item Several classes appearing in algebra admit intersections. For example \cite[Example 1.15]{grossberg2002}, let $K$ be the class of algebraically closed valued fields (we code the value group with an additional predicate), ordered by $F_1 \lea F_2$ if and only if $F_1$ is a subfield of $F_2$, the value groups are the same, and the valuations coincide on $F_1$. Then $K$ admits intersections. Again, $K$ is not universal (as it is not closed under substructure).
    \item The examples in \cite{non-locality} admit intersections. Since they are not $(<\aleph_0)$-tame, they cannot be universal classes (see Theorem \ref{pseudo-univ-tame}).
    \item The Hart-Shelah example \cite{hs-example, bk-hs} admits intersections but is also not $(<\aleph_0)$-tame.
    \item If $\mathcal{C}$ is a quasiminimal excellent pregeometry class (see \cite{zil05, quasimin}) then the AEC $K$ that it induces admits intersections and is categorical in every uncountable cardinal. Moreover it will be fully tame and short (at least assuming the existence of large cardinals). However it need not be finitary (take $\mathcal{C}$ to be the class of pseudo-exponential fields \cite[Theorem 2]{kirby-note-axiom}).
  \end{enumerate}
\end{example}

In the rest of this section, we give several equivalent definitions of admitting intersections and deduce some properties of these classes. All throughout this paper, we assume:

\begin{hypothesis}
  $K$ is an AEC.
\end{hypothesis}

\begin{defin}
  Let $M \in K$ and let $A \subseteq |M|$ be a set. $M$ is \emph{minimal over $A$} if whenever $M' \lea M$ contains $A$, then $M' = M$. $M$ is \emph{minimal over $A$ in $N$} if in addition $M \lea N$.
\end{defin}

\begin{defin}
  Let $N \in K$. We say $\F$ is a \emph{set of Skolem functions for $N$} if:

  \begin{enumerate}
    \item $\F$ is a non-empty set, and each element $f$ of $\F$ is a function from $N^{n}$ to $N$, for some $n < \omega$.
    \item For all $A \subseteq |N|$, $M := \F[A] := \bigcup\{ f[A] \mid f \in \F\}$ is such that $M \lea N$ and contains $A$.
  \end{enumerate}
\end{defin}

\begin{remark}
The proof of Shelah's presentation theorem (see \cite[Lemma I.1.7]{shelahaecbook}) gives that for each $N \in K$, there is $\F$ a set of Skolem functions for $N$ with $|\F| \le \LS (K)$.
\end{remark}

\begin{thm}\label{inter-charact}
  Let $K$ be an AEC and let $N \in K$. The following are equivalent:

  \begin{enumerate}
    \item\label{inter-charact-1} $N$ admits intersections.
    \item\label{inter-charact-2} There is an operator $\cl := \cl^N : \mathcal{P}(|N|) \rightarrow \mathcal{P} (|N|)$ such that for all $A, B \subseteq |N|$ and all $M \lea N$:

      \begin{enumerate}
        \item\label{cl-a} $\cl (A) \lea N$.
        \item\label{cl-b} $A \subseteq \cl (A)$.
        \item\label{cl-c} $A \subseteq B$ implies $\cl (A) \subseteq \cl (B)$.
        \item\label{cl-d} $\cl (M) = M$.
      \end{enumerate}
    \item\label{inter-charact-3} For each $A \subseteq |N|$, there is a unique minimal model over $A$ in $N$.
    \item\label{inter-charact-4} There is a set $\F$ of Skolem functions for $N$ such that:

      \begin{enumerate}
        \item $|\F| \le \LS (K)$.
        \item For all $M \lea N$, we have $\F[M] = M$.
      \end{enumerate}
  \end{enumerate}

  Moreover the operator $\cl^N : \mathcal{P} (|N|) \rightarrow \mathcal{P} (|N|)$ with the properties in (\ref{inter-charact-2}) is unique and if it exists then it has the following characterizations:

  \begin{itemize}
    \item $\cl^N (A) = \bigcap \{M \lea N \mid A \subseteq |M|\}$.
    \item $\cl^N (A) = \F[A]$, for any set of Skolem functions $\F$ for $N$ such that $\F[M] = M$ for all $M \lea N$.
    \item $\cl^N (A)$ is the unique minimal model over $A$ in $N$.
  \end{itemize}
\end{thm}
\begin{proof} \
  \begin{itemize}
    \item \underline{(\ref{inter-charact-1}) implies (\ref{inter-charact-2})}: Let $\cl^N (A) := \bigcap \{M \lea N \mid A \subseteq |M|\}$. Even without hypotheses on $N$, (\ref{cl-b}), (\ref{cl-c}), and (\ref{cl-d}) are satisfied. Since $N$ admits intersections, (\ref{cl-a}) is also satisfied.
    \item \underline{(\ref{inter-charact-2}) implies (\ref{inter-charact-3})}: Let $A \subseteq |N|$. Let $\cl$ be as given by (\ref{inter-charact-2}). Let $M := \cl (A)$. By (\ref{cl-a}), $M \lea N$. By (\ref{cl-b}), $A \subseteq |M|$. Moreover if $M' \lea N$ contains $A$, then by (\ref{cl-c}), $|M| \subseteq |\cl(M')|$ but by (\ref{cl-d}), $\cl(M') = M'$. Thus by coherence and (\ref{cl-a}) $M \lea M'$. This shows both that $M$ is minimal over $A$ and that it is unique.
    \item \underline{(\ref{inter-charact-3}) implies (\ref{inter-charact-4})}:
      We slightly change the proof of \cite[Lemma I.1.7]{shelahaecbook} as follows: Let $\chi := \LS (K)$. For each $\ba \in \fct{<\omega}{|N|}$, let $\seq{b_i^{\ba} : i < \chi}$ be an enumeration (possibly with repetitions) of the unique minimal model over $\text{ran}(\ba)$ in $N$. For each $n < \omega$ and $i < \chi$, we let $f_i^n : N^n \rightarrow N$ be $f_i^n (\ba) := b_i^{\ba}$. Let $\F := \{f_i^n \mid i < \chi, n < \omega\}$. Then $|\F| \le \LS (K)$ and if $A \subseteq |N|$, we claim that $\F[A]$ is minimal over $A$ in $N$. This shows in particular that $\F$ is as required.
      
      Let $M := \F[A]$. By definition, $M = \bigcup_{\ba \in \fct{<\omega}{|A|}} \F[\text{ran}(\ba)]$. Now if $\ba \in \fct{<\omega}{A}$, $M_{\ba} := \F[\text{ran}(\ba)] = \{b_i^{\ba} : i < \chi\}$ is the unique minimal model over $\text{ran} (\ba)$ in $N$. Thus if $\text{ran}(\ba) \subseteq \text{ran} (\bb)$, we must have (by coherence) $M_{\ba} \lea M_{\bb}$. It follows that $M \in K$ and by the axioms of AECs also $M \lea N$. Of course, $M$ contains $A$. Now if $M' \lea M$ contains $A$, then for all $\ba \in \fct{<\omega}{A}$, $\ba \in \fct{<\omega}{|M'|}$, so as $M_{\ba}$ is minimal over $\text{ran}(\ba)$, $M_{\ba} \lea M'$. It follows that $M \lea M'$ so $M = M'$.
    \item \underline{(\ref{inter-charact-4}) implies (\ref{inter-charact-1})}:
      Let $\F$ be as given by (\ref{inter-charact-4}). Let $A \subseteq |N|$. Let $M := \F[A]$. By definition of Skolem functions, $M$ contains $A$ and $M \lea N$. We claim that $M = \bigcap \{M' \lea N \mid A \subseteq |M'|\}$. Indeed, if $M' \lea N$ contains $A$, then by the hypothesis on $\F$, $M = \F[A] \subseteq \F[M'] = M'$. 
  \end{itemize}

  The moreover part follows from the arguments above.
\end{proof}

\begin{defin}
  For $N \in K$ let $\cl^N : \mathcal{P} (|N|) \rightarrow \mathcal{P} (|N|)$ be defined by $\cl^N (A) := \bigcap \{M \lea N \mid A \subseteq |M|\}$.
\end{defin}
\begin{remark}\label{univ-class-closure}
  If $K$ is a universal class, then one can take the set $\F$ of Skolem functions in (\ref{inter-charact-4}) to consist of all the functions of $N$. Thus $\cl^N (A)$ is just the closure of $A$ under the functions of $N$.
\end{remark}

Theorem \ref{inter-charact} allows us to deduce several properties of the operator $\cl^N$.

\begin{prop}\label{cl-props} 
  Let $M \lea N \in K$ and let $A, B \subseteq |N|$.
      \begin{enumerate}
        \item Invariance: If $f: N \cong N'$, then $f[\cl^N (A)] = \cl^{N'} (f[A])$.
        \item Monotonicity 1: $A \subseteq \cl^N (A)$.
        \item Monotonicity 2: $A \subseteq B$ implies $\cl^N (A) \subseteq \cl^N (B)$.
        \item Monotonicity 3: If $A \subseteq |M|$, then $\cl^N (A) \subseteq \cl^M (A)$. Moreover if $N$ admits intersections over $A$, then $M$ admits intersections over $A$ and $\cl^N (A) = \cl^M (A)$.
        \item Idempotence: $\cl^N(M) = M$.
        \item Finite character: If $N$ admits intersections, then if $B \subseteq \cl^N (A)$ is finite, there exists a finite $A_0 \subseteq A$ such that $B \subseteq \cl^N (A_0)$.

      \end{enumerate}
\end{prop}
\begin{proof}
Straightforward given Theorem \ref{inter-charact}: For finite character, use the characterization in terms of Skolem functions. For monotonicity 3, let $M_0 := \cl^N (A)$. We have $M_0 \lea N$ since $N$ admits intersections over $A$. Since $M \lea N$ contains $A$, we must have $|M_0| \subseteq |M|$. By coherence, $M_0 \lea M$, and by minimality, $M_0 = \cl^M (A)$.
\end{proof}

Note in particular the following:

\begin{cor} \
  \begin{enumerate}
    \item Assume that for every $M \in K$ and every $A \subseteq |M|$, there is $N \gea M$ such that $N$ admits intersections over $A$. Then $K$ admits intersections.
    \item $N \in K$ admits intersections if and only if it admits intersections over every finite $A \subseteq |N|$.
  \end{enumerate}
\end{cor}
\begin{proof} \
  \begin{enumerate}
    \item By Monotonicity 3.
    \item By the proof of Theorem \ref{inter-charact}.
  \end{enumerate}
\end{proof}
\begin{remark}
  The second result is implicit in the discussion after Remark 4.3 in \cite{non-locality}.
\end{remark}

Before stating the next proposition, we recall that any AECs admits a semantic notion of types. This was first introduced in \cite[Definition II.1.9]{sh300-orig}. We use the notation of \cite[Definition 2.16]{sv-infinitary-stability-afml}.

\begin{defin}[Galois types]\label{gtp-def} \
  \begin{enumerate}
    \item Let $K^3$ be the set of triples of the form $(\bb, A, N)$, where $N \in K$, $A \subseteq |N|$, and $\bb$ is a sequence of elements from $N$.
    \item For $(\bb_1, A_1, N_1), (\bb_2, A_2, N_2) \in K^3$, we write $(\bb_1, A_1, N_1)\Eat (\bb_2, A_2, N_2)$ if $A := A_1 = A_2$, and there exists $f_\ell : N_\ell \xrightarrow[A]{} N$ such that $f_1 (\bb_1) = f_2 (\bb_2)$. We call $\Eat$ \emph{atomic equivalence of triples} and say that two triples are \emph{atomically equivalent}.
    \item Note that $\Eat$ is a symmetric and reflexive relation on $K^3$. We let $E$ be the transitive closure of $\Eat$.
    \item For $(\bb, A, N) \in K^3$, let $\gtp (\bb / A; N) := [(\bb, A, N)]_E$. We call such an equivalence class a \emph{Galois type}.
    \item For $p = \gtp (\bb / A; N)$ a Galois type, define\footnote{It is easy to check that this does not depend on the choice of representatives.} $\ell (p) := \ell (\bb)$ and $\dom{p} := A$.
    \item We say a Galois types $p = \gtp (\bb / A; N)$ is \emph{algebraic} if $\bb \in \fct{\ell (\bb)}{A}$ (it is easy to check this does not depend on the choice of representatives). We mostly use this when $\ell (p) = 1$.
    \item For $N \in K$, $A \subseteq |N|$, and $\alpha$ an ordinal, we let $\gS^\alpha (A; N) := \{\gtp (\bb / A; N) \mid \bb \in \fct{\alpha}{|N|}\}$. When $\alpha = 1$, we omit it. For $M \in K$, we write $\gS^\alpha (M)$ for $\bigcup_{M' \gea M} \gS^\alpha (M; M')$. We similarly define $\gS^{<\infty} (M)$, etc.
  \end{enumerate}
\end{defin}

We can go on to define the restriction of a type (if $A_0 \subseteq \dom{p}$, $I \subseteq \ell (p)$, we will write $p^I \rest A_0$ when the realizing sequence is restricted to $I$ and the domain is restricted to $A_0$), the image of a type under an isomorphism, or what it means for a type to be realized. 

The next result says that in AECs admitting intersections, equality of Galois types is witnessed by an isomorphism. This can be seen as a weak version of amalgamation (see Section \ref{ap-sec}).

\begin{prop}\label{isom-eq-gtp}
  Assume $K$ admits intersections. Then $\gtp (\ba_1 / A; N_1) = \gtp (\ba_2 / A; N_2)$ if and only if there exists $f: \cl^{N_1} (A \ba_1) \cong_A \cl^{N_2} (A \ba_2)$ such that $f (\ba_1) = \ba_2$.
\end{prop}
\begin{proof} \
Let $M_1 := \cl^{N_1} (A \ba_1)$, $M_2 := \cl^{N_2} (A \ba_2)$. Since $N_\ell$ admits intersections, we have $M_\ell \lea N_\ell$, $\ell = 1,2$ so the right to left direction follows. Now assume $\gtp (\ba_1 / A; N_1) = \gtp (\ba_2 / A; N_2)$. It suffices to prove the result when the equality is atomic (then we can compose the isomorphisms in the general case). So let $N \in K$ and $f_\ell : N_\ell \xrightarrow[A]{} N$ witness atomic equality, i.e.\ $f_1 (\ba_1) = f_2 (\ba_2)$. By invariance and monotonicity 3, $f_\ell[M_\ell] = \cl^{f[N_\ell]}(A f_\ell (\ba_\ell)) = \cl^{N}(A f_\ell (\ba_\ell))$. Since $f_1 (\ba_1) = f_2 (\ba_2)$, we must have that $f_1[M_1] = f_2[M_2]$. Thus $f := (f_2 \rest M_2)^{-1} \circ (f_1 \rest M_1)$ is as desired.
\end{proof}

\begin{remark}
  Proposition \ref{isom-eq-gtp} was already observed (without proof) in \cite[Lemma 1.3]{non-locality}. Baldwin and Shelah also assert that $E = \Eat$ (see Definition \ref{gtp-def}), but this does not seem to follow.
\end{remark}

Before ending this section, we point out a technical disadvantage of the definition of admitting intersection. The notion is not closed under the tail of the AEC: If $K$ admits intersections and $\lambda$ is a cardinal, then it is not clear that $K_{\ge \lambda}$ admits intersections. Thus we will work with a slightly weaker definition:

\begin{defin}\label{k-m-def}
  For $K$ an AEC and $M \in K$, let $K_M$ be the AEC defined by adding constant symbols for the elements of $M$ and requiring that $M$ embeds inside every model of $K_M$. That is, $L(K_M) = L (K) \cup \{c_a \mid a \in |M|\}$, where the $c_a$'s are new constant symbols, and 

  $$
  K_M := \{(N, c_a^N)_{a \in |M|} \mid N \in K \text{ and }a \mapsto c_a^N \text{ is a }K\text{-embedding from } M \text{ into } N\}
  $$

  We order $K_M$ by $(N_1, c_a^N)_{a \in |M|} \lea (N_2, c_a^{N_2})$ if and only if $N_1 \lea N_2$ and $c_a^{N_1} = c_a^{N_2}$ for all $a \in |M|$.
\end{defin} 

\begin{defin}\label{local-def}
  For $P$ a property of AECs and $M \in K$, $K$ \emph{has $P$ above $M$} if $K_M$ has $P$. $K$ \emph{locally has $P$} if it has $P$ above every $M \in K$.
\end{defin}
\begin{remark}
  $K$ locally admits intersections if and only if for every $M \lea N$ in $K$ and every $A \subseteq |N|$ \emph{which contains $M$}, $\cl^N (A) \lea N$.
\end{remark}
\begin{remark}\label{local-rmk}
  If $K$ locally has $P$, then for every cardinal $\lambda$, $K_{\ge \lambda}$ locally has $P$.
\end{remark}

\section{Universal classes are fully tame and short}\label{tameness-sec}

In this section, we show that universal classes are fully $(<\aleph_0)$-tame and short. The basic argument for Theorem \ref{pseudo-univ-tame} is due to Will Boney and will also appear in \cite{tameness-groups}. 

Note that it is impossible to extend this result to AECs which admits intersections: \cite{non-locality} gives several counterexamples. One could hope that showing that categoricity in a high-enough cardinal implies tameness (a conjecture of Grossberg and VanDieren, see \cite[Conjecture 1.5]{tamenessthree}) is easier in AECs which admits intersections, but we have been unable to make progress in that direction and leave it to future work.

The key of the argument for tameness of universal classes is that the isomorphism characterizing the equality of Galois type is unique. We abstract this feature into a definition:

\begin{defin}
  $K$ is \emph{pseudo-universal} if it admits intersections and for any $N_1$, $N_2$, $\ba_1$, $\ba_2$, if $\gtp (\ba_1 / \emptyset; N_1) = \gtp (\ba_2 / \emptyset; N_2)$ and $f, g: \cl^{N_1} (\ba_1) \cong \cl^{N_2} (\ba_2)$ are such that $f (\ba_1) = g (\ba_1) =  \ba_2$, then $f = g$.
\end{defin}
\begin{example}\label{pseudouniv-example} \
  \begin{enumerate}
    \item\label{pseudouniv-1} In universal classes, $\cl^N (A)$ is just the substructure of $N$ generated by $A$ (see Remark \ref{univ-class-closure}). Thus universal classes are pseudo-universal.
    \item Let $K$ be the class of groups in the language containing only the multiplication symbol. Then $K$ is not a universal class but it is  pseudo-universal (the inverse function and the unit constant are first-order definable).
    \item We show below that pseudo-universal classes are $(<\aleph_0)$-tame, hence the AECs in \cite{non-locality} admit intersections but are not pseudo-universal.
    \item More simply, the class $\K$ of algebraically closed fields is elementary (hence $(<\aleph_0)$-tame), admits intersections, but is not pseudo-universal. Indeed, let $M$ be the algebraic closure of $\mathbb{Q}$ and let $x$ be transcendental. Let $N$ be the algebraic closure of $M \cup \{x\}$. Then there exists two different automorphisms of $N$ that fix $M \cup \{x\}$: the identity and one that sends $\sqrt{x}$ to $-\sqrt{x}$.
  \end{enumerate}
\end{example}

We quickly recall the definitions of tameness and shortness: Tameness as a property of AECs was introduced by Grossberg and VanDieren in \cite{tamenessone} (it was isolated from a proof in \cite{sh394}). It says that Galois types are determined by small restrictions of their domain. Shortness\footnote{What we call ``shortness'' is called ``type shortness'' by Boney, but in this paper we never write the ``type''.} was introduced by Will Boney in \cite[Definition 3.3]{tamelc-jsl}. It says that Galois types are determined by restrictions to small length. We will use the notation of \cite[Definition 2.22]{sv-infinitary-stability-afml}.

\begin{defin}
  Let $\kappa$ be an infinite cardinal.

  \begin{enumerate}
    \item $K$ is \emph{$(<\kappa)$-tame} if for any $M \in K$ and any $p \neq q$ in $\gS (M)$, there exists $A \subseteq |M|$ with $|A| < \kappa$ such that $p \rest A \neq q \rest A$.
    \item $K$ is \emph{fully $(<\kappa)$-tame and short} if for any $M \in K$, any ordinal $\alpha$, and any $p \neq q$ in $\gS^{\alpha} (M)$, there exists $I \subseteq \alpha$ and $A \subseteq |M|$ such that $|I| + |A| < \kappa$ and $p^I \rest A \neq q^I \rest A$.
    \item $\kappa$-tame means $(<\kappa^+)$-tame, similarly for fully $\kappa$-tame and short.
  \end{enumerate}
\end{defin}

\begin{defin}
  Let $\ba_\ell \in \fct{\alpha}{|N_\ell|}$ and let $\kappa$ be an infinite cardinal. We write $(\ba_1, N_1) \equiv_{<\kappa} (\ba_2, N_2)$ if for every $I \subseteq \alpha$ of size less than $\kappa$, $\gtp (\ba_1 \rest I / \emptyset; N_1) = \gtp (\ba_2 \rest I / \emptyset; N_2)$.
\end{defin}

The next proposition says roughly that it is enough to show shortness for types over the empty set. This appears already as \cite[Theorem 3.5]{tamelc-jsl}. We repeat the argument here for convenience.

\begin{prop}\label{short-tame}
  Let $\kappa$ be an infinite cardinal. Assume that for every $\alpha$, $N_\ell \in K$, $\ba_\ell \in \fct{\alpha}{|N_\ell|}$, $\ell = 1,2$, we have that $(\ba_1, N_1) \equiv_{<\kappa} (\ba_2, N_2)$ implies $\gtp (\ba_1 / \emptyset; N_1) = \gtp (\ba_2 / \emptyset; N_2)$. Then $K$ is fully $(<\kappa)$-tame and short.
\end{prop}
\begin{proof}
  Let $\beta$ be an ordinal, $M \in K$, $p, q \in \gS^{\beta} (M)$. Assume that $p^I \rest A = q^I \rest A$ for all $I \subseteq \beta$ and $A \subseteq |M|$ of size less than $\kappa$. Say $p = \gtp (\ba_1 / M; N_1)$, $q = \gtp (\ba_2 / M; N_2)$. Let $\bb$ be an enumeration of $|M|$ and let $p' := \gtp (\ba_1 \bb / \emptyset; N_1)$, $q' := \gtp (\ba_2 \bb / \emptyset; N_2)$. By assumption, $(p')^{I'} = (q')^{I'}$ for all $I'$ of size less than $\kappa$. In other words, $(\ba_1 \bb, N_1) \equiv_{<\kappa} (\ba_2 \bb, N_2)$. Therefore by our locality assumption $p' = q'$, and from the definition of Galois types this implies that $p = q$.
\end{proof}
\begin{remark}\label{loc-pseudouniv-rmk}
  By a similar argument, we can show that pseudo-universal classes are locally pseudo-universal (recall Definition \ref{local-def}).
\end{remark}

\begin{thm}\label{pseudo-univ-tame}
  If $K$ is pseudo-universal, then $K$ is fully $(<\aleph_0)$-tame and short.
\end{thm}
\begin{proof}
  Let $N_\ell \in K$, $\ba_\ell \in \fct{\alpha}{|N_\ell|}$, $\ell = 1,2$. Assume that $(\ba_1, N_1) \equiv_{<\aleph_0} (\ba_2, N_2)$. We show that $\gtp (\ba_1 / \emptyset; N_1) = \gtp (\ba_2 / \emptyset; N_1)$, which is enough by Proposition \ref{short-tame}. Let $M_\ell := \cl^{N_\ell} (\ran{\ba_\ell})$.

  For each finite $I \subseteq \alpha$, let $M_{\ell, I} := \cl^{N_\ell} (\ran{\ba_\ell \rest I})$. By definition of $\equiv_{<\aleph_0}$, for each finite $I \subseteq \alpha$, $\gtp (\ba_1 \rest I / \emptyset; N_1) = \gtp (\ba_2 \rest I / \emptyset; N_2)$. Therefore (because $K$ admits intersections) there exists $f_I : M_{1, I} \cong M_{2, I}$ such that $f_I (\ba_1 \rest I) = \ba_2 \rest I$. Moreover by definition of pseudo-universal, $f_I$ is unique with that property. This means in particular that if $I \subseteq J \subseteq \alpha$ are both finite, we must have $f_I \subseteq f_J$. By finite character of the closure operator, $M_\ell = \bigcup_{I \in [\alpha]^{<\aleph_0}} M_{\ell, I}$ and so letting $f := \bigcup_{I \in [\alpha]^{<\aleph_0}} f_I$, we have that $f: M_1 \cong M_2$ and $f (\ba_1) = \ba_2$. This witnesses that $\gtp (\ba_1 / \emptyset; M_1) = \gtp (\ba_2 / \emptyset; M_2)$ and so (since $M_\ell \lea N_\ell$), $\gtp (\ba_1 / \emptyset; N_1) = \gtp (\ba_2 / \emptyset; N_2)$.
\end{proof}
\begin{remark}
  If $K$ is a universal class (i.e.\ not ``pseudo''), then the proof of Theorem \ref{pseudo-univ-tame} (together with Remark \ref{univ-class-closure}) shows that Galois types are the same as quantifier-free types. That is, $\gtp (\ba_1 / A; N_1) = \gtp (\ba_2 / A; N_2)$ if and only if $\tp_{\text{qf}} (\ba_1 / A; N_1) = \tp_{\text{qf}} (\ba_2 / A; N_2)$, where $\tp_{\text{qf}} (\ba / A; N)$ denotes the quantifier-free types of $\ba$ over $A$ computed in $N$.
\end{remark}
We can localize Theorem \ref{pseudo-univ-tame} to obtain more generally (Remark \ref{loc-pseudouniv-rmk}):

\begin{cor}\label{loc-pseudo-univ-tame}
  If $K$ is locally pseudo-universal, then $K$ is fully $\LS (K)$-tame and short.
\end{cor}
\begin{proof}
  Let $M \in K$ and let $p, q \in \gS^{\alpha} (M)$. Assume that $p^I \rest A = q^I \rest A$ for all $A \subseteq |M|$ of at most size $\LS (K)$ and all $I \subseteq \alpha$ of size at most $\LS (K)$. We want to see that $p = q$. Without loss of generality $\|M\| \ge \LS (K)$. Let $M_0 \lea M$ have size $\LS (K)$. We know that $p^I \rest M_0 = q^I \rest M_0$ for all $I \subseteq \alpha$ of size at most $\LS (K)$. Since $K$ is locally pseudo-universal, $K_{M_0}$ (see Definition \ref{k-m-def}) is pseudo-universal. By Theorem \ref{pseudo-univ-tame}, $K_{M_0}$ is fully $(<\aleph_0)$-tame and short. Translating to $K$, this means that for any $N \gea M_0$, any $p', q' \in \gS^{\beta} (N)$, if $(p')^I \rest (|M_0| \cup A) = (q')^I \rest (|M_0| \cup A)$ for all finite $I$ and $A$, then $p' = q'$. Setting $N, p', q'$ to stand for $M, p, q$, we get that $p = q$, as desired.
\end{proof}

\section{Amalgamation from categoricity}\label{ap-sec}

We investigate how to get amalgamation from categoricity in tame AECs admitting intersections. In what follows, we will often use Remark \ref{high-enough-rmk} without comments. Recall:

\begin{defin}\label{ap-def}
  An AEC $K$ has amalgamation if for any $M_0 \lea M_\ell$, $\ell = 1,2$, there exists $N \in K$ and embeddings $f_\ell: M_\ell \xrightarrow[M_0]{} N$. We say that $K$ has $\lambda$-amalgamation if this holds for the models in $K_\lambda$. We define similarly \emph{disjoint amalgamation}, where we require in addition that $f_1[M_1] \cap f_2[M_2] = M_0$.
\end{defin}

We will use the concept of a good $\lambda$-frame, a notion of forking for types of length one over models of size $\lambda$, see \cite[Definition II.2.1]{shelahaecbook} or Appendix \ref{good-frame-appendix}. The following claim is a deep result of Shelah which says that good $\lambda$-frames exist in categorical classes. 

\begin{claim}\label{shelah-claim}
  If $K$ is categorical in unboundedly many cardinals, then there exists a categoricity cardinal $\lambda \ge \LS (K)$ such that $K$ has a good $\lambda$-frame (i.e.\ there exists a good $\lambda$-frame $\s$ such that $K_{\s} = K_\lambda$). In particular, $K$ has $\lambda$-amalgamation.
\end{claim}

The statement is implicit in Chapter IV of \cite{shelahaecbook}, but in June 2015 Will Boney and the author identified a gap in a key part of Shelah's proof \cite{categ-infinitary-v2}. In September 2016, Shelah communicated a fix to the author, which should appear as an online revision of Sh:734. As of December 2016, Shelah's fix has not yet been made public.

The key notion in the proof of Claim \ref{shelah-claim} is: 

\begin{defin}[Definition 2.1 in \cite{categ-infinitary-v2}]\label{event-synt-def}
  An AEC $K$ is \emph{$L_{\infty, \theta}$-syntactically characterizable} if whenever $M, N \in K$, if $M \lea N$ then $M \lee_{L_{\infty, \theta}} N$. We say that $K$ is \emph{eventually syntactically characterizable} if for every infinite cardinal $\theta$, there exists $\lambda$ such that $K_{\ge \lambda}$ is $L_{\infty, \theta}$-syntactically characterizable.
\end{defin}
\begin{remark}
  Using that saturated models are model-homogeneous, it is easy to see that any AEC \emph{with amalgamation} categorical in a proper class of cardinals is eventually syntactically characterizable \cite[Proposition 1.3]{categ-infinitary-v2}.
\end{remark}

The problematic part of Shelah's proof is a claim that an AEC categorical in unboundedly many cardinals is eventually syntactically characterizable (see \cite[Conclusion IV.2.14]{shelahaecbook}). However the following weakening is true: 

\begin{fact}[Conclusion IV.2.12.(1) in \cite{shelahaecbook}]\label{event-synt-fact} 
   If $K$ is categorical in cardinals of arbitrarily high cofinality (that is, for every $\theta$ there exists $\lambda$ such that $K$ is categorical in $\lambda$ and $\cf{\lambda} \ge \theta$), then $K$ is eventually syntactically characterizable.
\end{fact}

From an eventually syntactically characterizable AEC that is categorical in unboundedly many cardinals, Shelah's proof of Claim \ref{shelah-claim} goes through:

\begin{fact}[Theorem 2.12 in \cite{categ-infinitary-v2}]\label{good-frame-fact}
  If $K$ is eventually syntactically characterizable and categorical in unboundedly many cardinals, then there exists a categoricity cardinal $\lambda \ge \LS (K)$ such that $K$ has a good $\lambda$-frame.
\end{fact}

Thus it is reasonable to assume that we have a good $\lambda$-frame, and we want to transfer amalgamation above it. Our inspiration is a recent result of Adi Jarden, presented at a talk in South Korea in the Summer of 2014.

\begin{fact}[Corollary 7.16 in \cite{jarden-tameness-apal}]
  Assume $K$ has a good $\lambda$-frame where the class of uniqueness triples satisfies the existence property and $K$ is strongly $\lambda$-tame, then $K$ has $\lambda^+$-amalgamation.
\end{fact}

We will not give the definition of the class of uniqueness triples here (but see Definition \ref{dom-def} and Fact \ref{uq-domin}). It suffices to say that they are a version of domination for good frames. As for strong tameness, it is a variation of tameness relevant when amalgamation fails to hold. Recall that $\lambda$-tameness asks for two types that are equal on all their restrictions of size $\lambda$ to be equal. The strong version asks them to be \emph{atomically equal}, i.e.\ there is a map witnessing it that amalgamates the two models in which the types are computed, see Definition \ref{gtp-def}. Jarden's result is interesting, since it shows that tameness, a locality property that we see as quite mild compared to assuming amalgamation, can be of some use to proving amalgamation. The downside is that we have to ask for a strengthened version.

While Jarden proved much more than $\lambda^+$-amalgamation, it has been pointed out by Will Boney (in a private communication) that if one only wants amalgamation, the hypothesis that uniqueness triples satisfy the existence property is not necessary. The reason is that the argument of \cite{ext-frame-jml} can be used to transfer enough of the good frame to $\lambda^+$ so that the extension property holds, and the extension property implies amalgamation. 

We make the argument precise here and also show that less than strong tameness is needed (in particular, it suffices to assume tameness and that the AEC admits intersections). We first fix some notation.

\begin{defin}\label{gtp-def-0} \
  Let $\lambda \ge \LS (K)$.
  \begin{enumerate}
  \item $K^{3, 1}$ is the set of triples $(a, M, N)$ such that $M \lea N$ and $a \in N$. $K_\lambda^{3, 1}$ is the set of such triples where the models are in $K_\lambda$ (the difference with Definition \ref{gtp-def} is that we require the base to be a model and the sequence $\bb$ to have length one).
  \item We say $(a_1, M_1, N_1) \in K^{3, 1}$ \emph{atomically extends} $(a_0, M_0, N_0) \in K^{3,1}$ if $M_1 \gea M_0$ and $(a_1, M_0, N_1)\Eat (a_0, M_0, N_0)$ (recall Definition \ref{gtp-def})
  \item We say $M \in K_\lambda$ has the \emph{type extension property} if for any $N \gea M$ in $K_\lambda$ and any $p \in \gS (M)$, there exists $q \in \gS (N)$ extending $p$. 
  \item We say $M$ has the \emph{strong type extension property} if for any $N \gea M$, whenever $(a, M, M') \in K_\lambda^{3, 1}$, there exists $(b, N, N') \in K_\lambda^{3, 1}$ atomically extending $(a, M, M')$.
  \end{enumerate}

  We say \emph{$K_\lambda$ has the [strong] type extension property} (or \emph{$K$ has the [strong] type extension property in $\lambda$}) if every $M \in K_\lambda$ has it.
\end{defin}
\begin{remark}
  It is well-known (see for example \cite{grossbergbook}) that if $K$ has amalgamation, then $E = \Eat$. Similarly if $\lambda \ge \LS (K)$ and $K$ has $\lambda$-amalgamation, then $E \rest K_\lambda^{3, 1} = \Eat \rest K_\lambda^{3, 1}$. Moreover, $K$ has $\lambda$-amalgamation if and only if $K_\lambda$ has the strong type extension property.
\end{remark}

We think of the type extension property as saying that amalgamation cannot fail because there are ``fundamentally incompatible'' elements in the two models we want to amalgamate. Rather, the reason amalgamation fails is because we simply ``do not have enough models'' to witness that two types are equal in one step. It would be useful to formalize this intuition but so far we have failed to do so. 

We are interested in conditions implying that the type extension property (not the strong one) is enough to get amalgamation. For this, it turns out that it is enough to require that the AEC admits intersections. However we can even require a weaker condition:

\begin{defin}[Weak atomic equivalence]
  Let $(a_\ell, M, N_\ell) \in K_\lambda^{3, 1}$, $\ell = 1,2$. We say $(a_1, M, N_1) \Eat^- (a_2, M, N_2)$ (in words, $(a_1, M, N_1)$ and $(a_2, M, N_2)$ are \emph{weakly atomically equivalent}) if for $\ell = 1,2$, there exists $N_\ell' \lea N_\ell$ containing $a_\ell$ and $M$ such that $(a_\ell, M, N_\ell') \Eat (a_{3 - \ell}, M, N_{3 - \ell})$.
\end{defin}
\begin{defin}\label{weak-ap-def}
  $K$ has \emph{weak amalgamation} if $E \rest K^{3,1} = \Eat^- \rest K^{3,1}$, i.e.\ equivalence of triples is the same as \emph{weak} atomic equivalence of triples. Similarly define what it means for $K$ to have weak $\lambda$-amalgamation.
\end{defin}
\begin{remark}
  $K$ has weak amalgamation if and only if whenever $\gtp (a_1 / M; N_1) = \gtp (a_2 / M; N_2)$, there exists $N_1' \lea N_1$ containing $a_1$ and $M$ and there exists $N \gea N_2$ and $f: N_1' \xrightarrow[M]{} N$ so that $f (a_1) = a_2$.
\end{remark}
\begin{remark}\label{inter-implies-weak-ap}
  If $K$ locally admits intersections, $(a_\ell, M, N_\ell) \in K_\lambda^{3, 1}$, $\ell = 1,2$ and $(a_1, M, N_1) E (a_2, M, N_2)$, then by Proposition \ref{cl-props}, $N_\ell' := \cl^{N_\ell} (|M| \cup \{a_\ell\})$ witnesses that $(a_1, M, N_1)\Eat^- (a_2, M, N_2)$. Thus in that case, $E \rest K_\lambda^{3, 1} = \Eat^- \rest K_\lambda^{3, 1}$, so $K$ has weak amalgamation.
\end{remark}

Intuitively, weak amalgamation requires only that points that have the same Galois types can be amalgamated. The key result is:

\begin{thm}\label{thm-closure}
  Let $K$ be an AEC and $\lambda \ge \LS (K)$. Assume $K_\lambda$ has the type extension property. The following are equivalent:

  \begin{enumerate}
    \item $K$ has $\lambda$-amalgamation.
    \item $E \rest K_\lambda^{3, 1} = \Eat \rest K_\lambda^{3, 1}$ (i.e.\ equivalence of triples is the same as atomic equivalence of triples).
    \item $K$ has weak $\lambda$-amalgamation (i.e.\ equivalence of triples is the same as \emph{weak} atomic equivalence of triples).
  \end{enumerate}

  In particular, if $K$ admits intersections and has the type extension property, then it has amalgamation.
\end{thm}
\begin{proof}
  (1) implies (2) implies (3) is easy. We prove (3) implies (1). 

  Assume $E \rest K_\lambda^{3, 1} = \Eat^- \rest K_\lambda^{3, 1}$. The idea of the proof is as follows: we want to amalgamate a triple $(M_0, M, N)$, $M_0 \lea M$, $M_0 \lea N$. We use weak amalgamation first to amalgamate some smaller triple $(M_0, M', N')$ with $M_0 \lta M' \lea M$, $M_0 \lta N' \lea N$, then proceed inductively to amalgamate the entire triple. Claim 1 below shows that there exists a smaller triple which can be amalgamated and Claim 2 is a renaming of Claim 1. We then use Claim 2 repeatedly to amalgamate the full triple.

  \paragraph{\underline{Claim 1}} For every triple $(M_0, M_1, M_2)$ of models in $K_\lambda$ so that $M_0 \lta M_1$ and $M_0 \lea M_2$, there exists $M_1' \lea M_1$ and $M_2' \gea M_2$ in $K_\lambda$ such that $M_0 \lta M_1'$, and there exists $g: M_1' \xrightarrow[M_0]{} M_2'$.

  \[
  \xymatrix{
    M_1 &  \\
    M_1' \ar[u] \ar@{.>}[r]_g & M_2' \\
    M_0 \ar[u] \ar[r] & M_2 \ar@{.>}[u] \\
    }
  \]

  \paragraph{\underline{Proof of claim 1}}

  Let $M_0 \lta M_\ell$ be models in $K_\lambda$, $\ell = 1,2$. Pick any $a_1 \in |M_1| \backslash |M_0|$. Let $p := \gtp (a_1 / M_0; M_1)$. By the type extension property, there exists $q \in \gS (M_2)$ extending $p$. Pick $M_2^\ast \gea M_2$ and $a_2 \in |M_2^\ast|$ such that $q = \gtp (a_2 / M_2; M_2^\ast)$. Since $E$ is $\Eat^-$ over the domain of interest, we have $(a_1, M_0, M_1) \Eat^- (a_2, M_0, M_2^\ast)$. Let $M_1' \lea M_1$ contain $a_1$ and $M_0$ such that $(a_1, M_0, M_1') \Eat (a_2, M_0, M_2^\ast)$. By definition, we have that there exists $M_2' \gea M_2^\ast$ such that $M_1'$ embeds into $M_2'$ over $M_0$, as needed. $\dagger_{\text{Claim 1}}$

  Now we obtain amalgamation by repeatedly applying Claim 1. Since the result is key to subsequent arguments, we give full details below.

  \paragraph{\underline{Claim 2}} For every triple $(M_0, M_1, M_2)$ of models in $K_\lambda$ so that $M_0 \lta M_1$ and $f: M_0 \rightarrow M_2$, there exists $M_1' \lea M_1$, $M_2' \gea M_2$ in $K_\lambda$ and $g: M_1' \xrightarrow{} M_2'$ such that $M_0 \lta M_1'$ and $f \subseteq g$.

  \[
  \xymatrix{
    M_1 &  \\
    M_1' \ar[u] \ar@{.>}[r]_g & M_2' \\
    M_0 \ar[u] \ar[r]_f & M_2 \ar@{.>}[u] \\
    }
  \]

  \paragraph{\underline{Proof of claim 2}}

  Let $M_0, M_1, M_2$ and $f$ be as given by the hypothesis. Let $\widehat{M_2}$ and $\bigf$ be such that $f \subseteq \bigf$, $M_0 \lea \widehat{M_2}$ and $\bigf : \widehat{M_2} \cong M_2$. Now apply Claim 1 to $(M_0, M_1, \widehat{M_2})$ to obtain $M_1' \lea M_1$ with $M_0 \lta M_1'$, $\widehat{M_2}' \gea \widehat{M_2}$ and $\bigg: M_1' \xrightarrow[M_0]{} \widehat{M_2}'$. Now let $\bigf'$, $M_2'$ be such that $M_2' \gea M_2$ and $\bigf': \widehat{M_2}' \cong M_2'$ extends $\bigf$. Let $g := \bigf' \circ \bigg$. Since $\bigg$ fixes $M_0$ and $\bigf'$ extends $f$, $g$ extends $f$, as desired. $\dagger_{\text{Claim 2}}$

  Now let $M_0 \lea M$ and $M_0 \lea N$ be in $K_\lambda$. We want to amalgamate $M$ and $N$ over $M_0$. We try to build $\seq{M_i : i < \lambda^+}$, $\seq{N_i : i < \lambda^+}$ increasing continuous in $K_\lambda$ and $\seq{f_i : i < \lambda^+}$ an increasing continuous sequence of embeddings such that for all $i < \lambda^+$:

  \begin{enumerate}
    \item $M_i \lea M$.
    \item $f_i : M_i \xrightarrow[M_0]{} N_i$.
    \item $N_0 = N$.
    \item $M_i \lta M_{i + 1}$.
  \end{enumerate}

  This is impossible since then $\bigcup_{i < \lambda^+} M_i$ has cardinality $\lambda^+$ but is a $K$-substructure of $M$ which has cardinality $\lambda$. Now for $i = 0$, we can take $N_0 = N$ and $f_0 = \text{id}_{M_0}$ and for $i$ limit we can take unions. Therefore there must be some $\alpha < \lambda^+$ such that $f_\alpha$, $M_\alpha$, $N_\alpha$ are defined but we cannot define $f_{\alpha + 1}$, $M_{\alpha + 1}$, $N_{\alpha + 1}$. If $M_\alpha \lta M$, we can use Claim 2 (with $M_0$, $M_1$, $M_2$, $f$ there standing for $M_\alpha$, $M$, $N_\alpha$, $f_\alpha$ here) to get $M_{\alpha + 1} \lea M$ with $M_\alpha \lta M_{\alpha + 1}$ and $N_{\alpha + 1} \gea N_\alpha$ with $f_{\alpha + 1} : M_{\alpha + 1} \rightarrow N_{\alpha + 1}$ extending $f_\alpha$ (so $M_1'$, $M_2'$, $g$ in Claim 2 stand for $M_{\alpha + 1}$, $N_{\alpha + 1}$, $f_{\alpha + 1}$ here). Thus we can continue the induction, which we assumed was impossible. Therefore $M_\alpha = M$, so $f_\alpha : M \xrightarrow[M_0]{} N_\alpha$ amalgamates $M$ and $N$ over $M_0$, as desired.
\end{proof}
\begin{remark}\label{dense-basic-rmk}
  It is enough to assume that the type extension property holds on a set of types satisfying what Shelah calls the \emph{density}  property of basic types (see axiom (D)(c) in \cite[Definition II.2.1]{shelahaecbook}): for any $M \lta N$ in $K_\lambda$, there exists $b \in |N| \backslash |M|$ such that $\gtp (b / M; N)$ can be extended to any $M' \gea M$ with $M' \in K_\lambda$. This generalization is the reason the last part of the proof is done non-constructively rather than first enumerating $M$ and amalgamating it element by element. This is used to prove Theorem \ref{ap-frame-transfer} in full generality (i.e.\ without assuming that the good frame is type-full).
\end{remark}

We are now ready to formally state the amalgamation transfer:

\begin{thm}\label{ap-frame-transfer}
  Let $K$ be an AEC. Let $\lambda \ge \LS (K)$ and assume $\s$ is a good $\lambda$-frame with underlying class $K_\lambda$. If:

  \begin{enumerate}
    \item $K$ is $\lambda$-tame.
    \item $K_{\ge \lambda}$ has weak amalgamation.
  \end{enumerate}

  Then $K_{\ge \lambda}$ has amalgamation.
\end{thm} 
\begin{proof}
  We extend $\s$ to models of size greater than $\lambda$ by defining $\ge \s$ as in \cite[Section II.2]{shelahaecbook} (or see \cite[Definition 2.7]{ext-frame-jml}). Even without assuming tameness or weak amalgamation, Shelah has shown that $\ge \s$ has local character, density of basic types, and transitivity. Moreover, tameness implies that it has uniqueness. Now work by induction on $\mu \ge \lambda$ to show that $K$ has $\mu$-amalgamation. When $\mu = \lambda$ this follows from the definition of a good frame so assume $\mu > \lambda$. As in \cite[Theorem 5.13]{ext-frame-jml}, we can prove that $\ge \s$ has the extension property for models of size $\mu$ (the key is that the directed system argument only uses amalgamation below $\mu$). In particular, $K_{\mu}$ has the type extension property for basic types. The proof of Theorem \ref{thm-closure} together with the density of basic types (see Remark \ref{dense-basic-rmk}) shows that this suffices to get $\mu$-amalgamation.
\end{proof}

\begin{cor}\label{tame-ap-0}
  Let $K$ be a tame AEC that is eventually syntactically characterizable and categorical in unboundedly many cardinals. If $K$ has weak amalgamation, then there exists $\lambda$ such that $K_{\ge \lambda}$ has amalgamation.
\end{cor}
\begin{proof}
  By Fact \ref{good-frame-fact}, we can find $\lambda \ge \LS (K)$ such that $K_{\lambda}$ has a good frame and $K$ is $\lambda$-tame. By Theorem \ref{ap-frame-transfer}, $K_{\ge \lambda}$ has amalgamation.
\end{proof}

\begin{cor}\label{tame-ap}
  Let $K$ be an eventually syntactically characterizable AEC categorical in unboundedly many cardinals. If $K$ is tame and locally admits intersections, then there exists $\lambda$ such that $K_{\ge \lambda}$ has amalgamation.
\end{cor}
\begin{proof}
  By Remark \ref{inter-implies-weak-ap}, $K$ has weak amalgamation. Now apply Corollary \ref{tame-ap-0}.
\end{proof}

\begin{cor}\label{univ-ap}
  Let $K$ be locally pseudo-universal AEC. If $K$ is eventually syntactically characterizable and categorical in unboundedly many cardinals, then there exists $\lambda$ such that $K_{\ge \lambda}$ has amalgamation.
\end{cor}
\begin{proof}
  By Corollary \ref{loc-pseudo-univ-tame}, $K$ is tame. Now apply Corollary \ref{tame-ap}.
\end{proof}

We can apply these results to Shelah's categoricity conjecture and improve Fact \ref{ap-categ}. When $K$ has primes, this will be further improved in Section \ref{categ-transfer-sec}.

\begin{cor}\label{prev-cor}
  Let $K$ be a tame AEC with weak amalgamation.

  \begin{enumerate}
    \item\label{prev-cor-1} If $K$ is categorical in a high-enough successor cardinal, then $K$ is categorical on a tail of cardinals.
    \item Assume $2^{\theta} < 2^{\theta^+}$ for every cardinal $\theta$ and an unpublished claim of Shelah (Claim \ref{shelah-xxx}). If $K$ is eventually syntactically characterizable and categorical in unboundedly many cardinals, then $K$ is categorical on a tail of cardinals.
  \end{enumerate}
\end{cor}
\begin{proof}
  By Corollary \ref{tame-ap-0} (using Fact \ref{event-synt-fact} to see that $K$ is eventually syntactically characterizable in (\ref{prev-cor-1})), we can assume without loss of generality that $K$ has amalgamation. Now:

  \begin{enumerate}
    \item Apply \cite{tamenessthree} (and \cite{sh394} can also give a downward transfer).
    \item Apply Fact \ref{ap-categ}.
  \end{enumerate}
\end{proof}

Note that even if $K$ is a universal class which \emph{already} has amalgamation, Theorem \ref{ap-frame-transfer} is still key to transfer categoricity (see Theorem \ref{categ-transfer-frame}).

\section{Categoricity transfer in AECs with primes}\label{categ-transfer-sec}

In this section, we prove a categoricity transfer for AECs that have amalgamation and primes. Prime triples were introduced in \cite[Section III.3]{shelahaecbook}, see also \cite{jarden-prime}.

\begin{defin}\label{prime-defs} \
  \begin{enumerate}
    \item Let $M \in K$ and let $A \subseteq |M|$. $M$ is \emph{prime over $A$} if for any enumeration $\ba$ of $A$ and any $N \in K$, whenever $\gtp (\ba / \emptyset; M) = \gtp (\bb / \emptyset; N)$, there exists $f: M \rightarrow N$ such that $f (\ba) = \bb$.
    \item $(a, M, N)$ is a \emph{prime triple} if $M \lea N$, $a \in |N|$, and $N$ is prime over $|M| \cup \{a\}$.
    \item $K$ \emph{has primes} if for any $p \in \gS (M)$ there exists a prime triple $(a, M, N)$ such that $p = \gtp (a / M; N)$.
    \item $K$ \emph{weakly has primes} if whenever $\gtp (a_1 / M; N_1) = \gtp (a_2 / M; N_2)$, there exists $M_1 \lea M$ containing $a_1$ and $N_1$ and $f: M_1 \xrightarrow[M]{} N_2$ such that $f (a_1) = a_2$. Similarly define what it means for $K_\lambda$ to have or weakly have primes.
  \end{enumerate}
\end{defin}
\begin{remark}\label{prime-weakly-prime}
  For $M \lea N$ and $a \in |N|$, $(a, M, N)$ is a prime triple if and only if whenever $\gtp (b / M; N') = \gtp (a / M; N)$, there exists $f: N \xrightarrow[M]{} N'$ such that $f (a) = b$. Thus if $K$ has primes, then $K$ weakly has primes.
\end{remark}
\begin{remark}\label{inter-prime}
  If $K$ admits intersections, $M \lea N$, and $a \in |N|$, $(a, M, \cl^N (|M| \cup \{a\}))$ is a prime triple. Thus $K$ has primes.
\end{remark}

Assume $K$ is an AEC categorical in $\lambda := \LS (K)$ (this is a reasonable assumption as we can always restrict ourselves to the class of $\lambda$-saturated models of $K$). Our goal is to prove (with more hypotheses) that if $K$ is categorical in a $\theta > \lambda$ then it is categorical in all $\theta' \ge \lambda$. To accomplish this, we will show that $K_\lambda$ is \emph{uni-dimensional}. In \cite[Section III.2]{shelahaecbook}, Shelah gives several possible generalization of the first-order definition in \cite[Definition V.2.2]{shelahfobook}. We have picked what seems to be the most convenient to work with:

\begin{defin}[Definition III.2.2.6 in \cite{shelahaecbook}]
  Let $\lambda \ge \LS (K)$. $K_\lambda$ is \emph{weakly uni-dimensional} if for every $M \lta M_\ell$, $\ell = 1,2$ all in $K_\lambda$, there is $c \in |M_2| \backslash |M|$ such that $\gtp (c / M; M_2)$ has more than one extension in $\gS (M_1)$.
\end{defin}

To understand this definition, it might be helpful to look at the negation: there exists $M \lta M_\ell$, $\ell = 1,2$ all in $K_\lambda$ such that for all $c \in |M_2| \backslash |M|$, $\gtp (c / M; M_2)$ has exactly one extension in $\gS (M_1)$. Working in a good frame, this one extension must be the nonforking extension (so in particular $\gtp (c / M; M_2)$ is omitted in $M_1$). It turns out that for any $c \in |M_2| \backslash |M|$ and $d \in |M_1| \backslash |M|$, $\gtp (c / M; M_2)$ and $\gtp (d / M; M_1)$ are orthogonal (in a suitable sense, see Appendix \ref{proof-appendix}), so they will generate two different dimensions.

\begin{fact}[Claim III.2.3.(4) in \cite{shelahaecbook}]\label{unidim-categ}
  Let $\lambda \ge \LS (K)$. If $K_\lambda$ is weakly uni-dimensional, is categorical in $\lambda$, is stable in $\lambda$, and has $\lambda$-amalgamation, then\footnote{In \cite[Claim III.2.3.(4)]{shelahaecbook}, Shelah assumes more generally the existence of a good $\lambda$-frame, but the proof shows that the hypotheses mentioned here suffice. In any case, we will only use Fact \ref{unidim-categ} inside a good frame.} $K$ is categorical in $\lambda^+$.
\end{fact}

If $K$ is $\lambda$-tame and has amalgamation, then categoricity in $\lambda^+$ is enough by the categoricity transfer of Grossberg and VanDieren:

\begin{fact}[Theorem 6.3 in \cite{tamenessthree}]\label{tameness-categ}
  Assume $K$ is an $\LS (K)$-tame AEC with amalgamation and no maximal models. If $K$ is categorical in $\LS (K)$ and $\LS (K)^+$, then $K$ is categorical in all $\mu \ge \LS (K)$.
\end{fact}

 Thus the hard part is showing that $K_{\LS (K)}$ is weakly uni-dimensional. We proceed by contradiction.

 \begin{defin}[III.12.39.(d) in \cite{shelahaecbook}]\label{kneg-def}
    Let $M \in K$ and let $p \in \gS (M)$. We define\footnote{Shelah calls the class $K^\ast$.} $K_{\neg^\ast p}$ to be the class of $N \in K_M$ (recall Definition \ref{k-m-def}) such that $f(p)$ has a unique extension to $\gS(N \rest L (K))$. Here $f: M \rightarrow N$ is given by $f (a) := c_a^N$. We order $K_{\neg^\ast p}$ with the strong substructure relation induced from $K_M$.
\end{defin}
\begin{remark}
  Let $p \in \gS (M)$ be nonalgebraic and let $M \lea N$. If we are working in a good frame and $p$ has a unique extension to $\gS (N)$, then it must be the nonforking extension. Thus $p$ is omitted in $N$. However even if $p$ is omitted in $N$, $p$ could have two nonalgebraic extensions to $\gS (N)$, so $K_{\neg^\ast p}$ need \emph{not} be the same as the class $\K_{\neg p}$ of models omitting $p$.
\end{remark}

In general, we do \emph{not} claim that $K_{\neg^\ast p}$ is an AEC. Nevertheless it is an abstract class in the sense introduced by Grossberg in \cite{grossbergbook}, see \cite[Definition 2.7]{sv-infinitary-stability-afml}. Thus we can define notions such as amalgamation, Galois types, and tameness there just as in AECs. The following gives an easy criterion for when $K_{\neg^\ast p}$ \emph{is} an AEC:

\begin{prop}\label{kneg-aec}
  Let $\s = (K, \nf)$ be a type-full good $(\ge \lambda)$-frame (so $\lambda = \LS (K)$ and $K_{<\lambda} = \emptyset$). Let $M \in K$ and let $p \in \gS (M)$. Then $K_{\neg^\ast p}$ is an AEC.
\end{prop}
\begin{proof}
  All the axioms are easy except closure under chains. So let $\delta$ be a limit ordinal and let $\seq{N_i : i < \delta}$ be increasing continuous in $K_{\neg^\ast p}$. Identify models in $K$ with their expansions in $K_M$, assuming without loss of generality that $M \lea N_0$, i.e.\ the map $a \mapsto c_a^{N_0}$ for $a \in M$ is the identity. Let $N_\delta := \bigcup_{i < \delta} M_i$. We have that $N_\delta \rest L(K) \in K$. Now if $p_1, p_2 \in \gS (N_\delta \rest L(K))$ are two extensions of $p$, by local character there exists $i < \delta$ such that $p_1$ and $p_2$ do not fork over $N_i$. Since $p$ has a unique extension to $N_i$, $p_1 \rest N_i = p_2 \rest N_i$. By uniqueness, $p_1 \rest N_\delta = p_2 \rest N_\delta$.
\end{proof}

In fact, Shelah gave a criterion for when $K_{\neg^\ast p}$ has a good $\lambda$-frame:

\begin{fact}[Claim III.12.39 in \cite{shelahaecbook}]\label{not-unidim-frame}
  Let $\s$ be a good $\lambda$-frame with underlying class $K_\lambda$. Assume $\s$ is type-full, $\goodp$, successful (see appendix \ref{good-frame-appendix} for the definitions of these terms), and $K_\lambda$ has primes. Assume further that $K$ is categorical in $\lambda$.
  
  If $K_\lambda$ is not weakly uni-dimensional, then there exists $M \in K_\lambda$ and $p \in \gS (M)$ such that $\s \rest K_{\neg^\ast p}$ (the restriction of $\s$ to models in $K_{\neg^\ast p}$) is a type-full good $\lambda$-frame.
\end{fact}

Since this result is crucial to our argument and Shelah's proof is only implicit, we have included a proof in Appendix \ref{proof-appendix}.

Note that the hypotheses of Fact \ref{not-unidim-frame} are reasonable. In fact, it is known that they follow from categoricity in fully tame and short AECs with amalgamation:

\begin{fact}[Theorem 15.6 in \cite{indep-aec-apal}]\label{known-good-frame}
  Let $K$ be a fully $(<\kappa)$-tame and short AEC with amalgamation. Let $\lambda$, $\mu$ be cardinals such that:

  $$
  \LS (K) < \kappa = \beth_\kappa < \lambda = \beth_\lambda \le \mu
  $$

  Assume further that $\cf{\lambda} \ge \kappa$. If $K$ is categorical in $\mu$, then $K$ is categorical in $\lambda$ and there exists a type-full successful good $\lambda$-frame $\s$ with underlying class $K_\lambda$.
\end{fact}

From Proposition \ref{indep-goodp}, it will follow that the frame given by Fact \ref{known-good-frame} is also $\goodp$. If in addition the AEC has primes (e.g. if it is universal), then the hypotheses are satisfied. Of course, the Hanf numbers in Fact \ref{known-good-frame} are not optimal. We give the following improvement in Appendix \ref{good-frame-appendix}:

\begin{thm}\label{get-good-frame}
  Let $K$ be a fully $\LS (K)$-tame and short AEC with amalgamation and no maximal models. If $K$ is categorical in a $\mu > \LS (K)$, then there exists $\lambda_0 < \hanf{\LS (\K)}$ such that for all $\lambda \ge \lambda_0$ where $K$ is categorical in $\lambda$, there exists a type-full successful $\goodp$ $\lambda$-frame with underlying class $K_\lambda$.
\end{thm}
\begin{proof}
  Combine Corollaries \ref{fully-good-cor} and \ref{get-good-frame-cor}.
\end{proof}

Now we reach a crucial point. For the purpose of a categoricity transfer, it would be enough to show that $K_{\neg^\ast p}$ above has arbitrarily large models, since this means that there are non-saturated models in every cardinal above $\lambda$. Unfortunately, even if $K$ is fully tame and short and has amalgamation, it is not easy to get a handle on $K_{\neg^\ast p}$. For example, it is not clear if it has amalgamation or even if it is tame. In \cite[Discussion III.12.40]{shelahaecbook} Shelah claims to be able to show using enough instances of the weak  generalized continuum hypothesis that $\s \rest K_{\neg^\ast p}$ above has arbitrarily large models (this is probably how Claim \ref{shelah-xxx} is proven) and this is the key to the proof of Fact \ref{ap-categ}.

We make the situation where $K_{\neg^\ast p}$ is well-behaved into a definition:

\begin{defin}\label{nice}
  $K$ is \emph{nice} if:
  \begin{enumerate}
    \item $K$ has weak amalgamation. 
    \item For any $M \in K$ and any $p \in \gS (M)$, $K_{\neg^\ast p}$ has weak amalgamation and if $K$ is $\|M\|$-tame, then so is $K_{\neg^\ast p}$.
  \end{enumerate}
\end{defin} 

Note that if $K$ is a universal class, then $K_{\neg^\ast p}$ also is universal (using that $K$ is fully $(<\aleph_0)$-tame and short, we can prove as in Proposition \ref{kneg-aec} that it is an AEC), hence $K$ is nice! More generally:

\begin{prop}\label{weakly-prime-nice}
  If $K$ weakly has primes, then $K$ is nice.
\end{prop}
\begin{proof}
  Weak amalgamation follows from the definition of weakly having primes. Now let $M \in K$ and $p \in \gS (M)$. Observe that $K_{\neg^\ast p}$ weakly has primes, because if $N \in K_{\neg^\ast p}$, $N_0 \lea N \rest L (K)$ is in $K$, and $M \lea N_0$, then the natural expansion of $N_0$ is in $K_{\neg^\ast p}$. Therefore $K_{\neg^\ast p}$ also has weak amalgamation. If in addition $K$ is $\|M\|$-tame, then so is $K_{\neg^\ast p}$: indeed if $N \in K_{\neg^\ast p}$, $q_1, q_2 \in \gS (N)$, and the two types are equal in $K$, then since $K_{\neg^\ast p}$ weakly has primes there is a map witnessing equality of the types in $K_{\neg^\ast p}$ also.
\end{proof}

The following fact is the key to our argument. It was first proven under slightly stronger hypotheses by Will Boney \cite{ext-frame-jml}. The interesting consequence to us is that it gives a local criterion for a tame AEC to have arbitrarily large models.

\begin{fact}[Corollary 6.10 in \cite{tame-frames-revisited-v6-toappear}]\label{nmm-fact}
  If $\s$ is a good $\lambda$-frame on $K_\lambda$, $K$ is $\lambda$-tame and has amalgamation, then $\s$ extends to a good $(\ge \lambda)$-frame on $K_{\ge \lambda}$. In particular, $K_{\ge \lambda}$ has no maximal models and is stable in every cardinal above $\lambda$.
\end{fact}

\begin{thm}\label{categ-transfer-frame} 
  Let $\s$ be a good $\lambda$-frame with underlying AEC $K$. Assume $\s$ is type-full, $\goodp$, successful, and $K_\lambda$ has primes. Assume also that $K$ is categorical in $\lambda$, $\lambda$-tame, and nice. The following are equivalent.

  \begin{enumerate}
    \item\label{categ-transfer-1} $K_\lambda$ is weakly uni-dimensional.
    \item\label{categ-transfer-2} $K$ is categorical in all $\mu \ge \lambda$.
    \item\label{categ-transfer-3} $K$ is categorical in some $\theta > \lambda$.
  \end{enumerate}
\end{thm}
\begin{proof} 
  Replacing $K$ with $K_{\ge \lambda}$, assume without loss of generality that $\lambda = \LS (K)$ and $K_{<\LS (K)} = \emptyset$. First note that $K$ has amalgamation by Theorem \ref{ap-frame-transfer}. By Fact \ref{nmm-fact}, $\s$ extends to a good $(\ge \lambda)$-frame on $K$. In particular, $K$ has no maximal models and is stable in every cardinal. Moreover by Proposition \ref{kneg-aec}, $K_{\neg^\ast p}$ is an AEC for all $p \in \gS (M)$ and $M \in K$.

  If $K_\lambda$ is weakly uni-dimensional, then by Fact \ref{unidim-categ}, $K$ is categorical in $\lambda^+$. By Fact \ref{tameness-categ}, $K$ is categorical in all $\mu \ge \lambda$. So (\ref{categ-transfer-1}) implies (\ref{categ-transfer-2}). Of course, (\ref{categ-transfer-2}) implies (\ref{categ-transfer-3}). It remains to show (\ref{categ-transfer-3}) implies (\ref{categ-transfer-1}). We show the contrapositive. 

  Assume that $K$ is \emph{not} weakly uni-dimensional. Let $M \in K_\lambda$ and $p \in \gS (M)$ be as given by Fact \ref{not-unidim-frame}. Let $\s_{\neg^\ast p} := \s \rest K_{\neg^\ast p}$, the restriction of $\s$ to models in $K_{\neg^\ast p}$. Since $K$ is nice, $K_{\neg^\ast p}$ has weak amalgamation and since $K$ is also $\lambda$-tame, $K_{\neg^\ast p}$ is $\lambda$-tame. Since $\s_{\neg^\ast p}$ is a good $\lambda$-frame, Theorem \ref{ap-frame-transfer} gives that $K_{\neg^\ast p}$ has amalgamation. By Fact \ref{nmm-fact}, $K_{\neg^\ast p}$ has no maximal models and is stable in every cardinals. Now let $\theta > \lambda$. By stability, $K$ has a saturated model of size $\theta$. Moreover since $K_{\neg^\ast p}$ has arbitrarily large models there must exist $\bigN \in K_{\neg^\ast p}$ of size $\theta$. By construction, $\bigN \rest L (K)$ is not saturated of size $\theta$. Therefore $K$ is not categorical in $\theta$.
\end{proof}

We are now ready to prove a categoricity transfer in fully tame and short AECs with amalgamation (Theorem \ref{prime-caract-0} from the abstract). We state one more fact:

\begin{fact}\label{baldwin-omit}
  If $K$ is a $\LS (K)$-tame AEC with amalgamation and no maximal models which is categorical in a $\lambda \ge H_2$ (recall Notation \ref{hanf-notation}) and the model of size $\lambda$ is saturated, then $K$ is categorical in $H_2$.
\end{fact}
\begin{proof}
  By the proof of \cite[Theorem II.1.6]{sh394} (or see \cite[Theorem 14.8]{baldwinbook09}).
\end{proof}

\begin{thm}\label{categ-nice-ap}
  Let $K$ be a fully $\LS (K)$-tame and short AEC with amalgamation such that $K_{\ge H_2}$ has primes. If $K$ is categorical in a $\lambda > H_2$, then $K$ is categorical in all $\lambda' \ge H_2$.
\end{thm}
\begin{proof}
  Without loss of generality, $K$ has joint embedding and no maximal models: we can start by splitting $K$ into disjoint parts, each of which has joint embedding, and then work with the unique part which has arbitrarily large models.

  We start by observing that $K$ is categorical in $H_2$ by Fact \ref{baldwin-omit} (note that the model of size $\lambda$ is saturated by Facts \ref{frame-construct} and \ref{shelah-vi}). Now apply Theorem \ref{get-good-frame} (to $K_{\ge \LS (K)^+}$) and Theorem \ref{categ-transfer-frame}.
\end{proof}

The only place where shortness is used above is to get the existence property for uniqueness triples (i.e.\ that the good frame is successful). The proof shows that it is enough to assume that for some $\lambda$, $K_{\ge \lambda}$ is almost fully good, i.e.\ it has a nice-enough global independence relation (see \ref{fully-good-def} for a more precise definition). One can ask:

\begin{question}\label{tame-shortness-q}
  Can the full tameness and shortness hypothesis be weakened to just being $\LS (K)$-tame? 
\end{question}

We obtain a categoricity transfer for universal classes with amalgamation.

\begin{cor}\label{categ-univ-ap}
  Let $K$ be a locally pseudo-universal AEC with amalgamation. If $K$ is categorical in a $\lambda > H_2$, then $K$ is categorical in all $\lambda' \ge H_2$.
\end{cor}
\begin{proof}
  By Corollary \ref{loc-pseudo-univ-tame}, $K$ is fully $\LS (K)$-tame and short. By Remark \ref{inter-prime}, $K$ has primes. Now apply Theorem \ref{categ-nice-ap}.
\end{proof}

In view of Theorem \ref{categ-nice-ap}, a natural question is whether the existence of primes follows from the other hypotheses: 

\begin{question}\label{weakly-primes}
  If $K$ is fully tame and short, has amalgamation, and is categorical in unboundedly many cardinals, does there exists $\lambda$ such that $K_{\ge \lambda}$ has primes?
\end{question}

Note that by \cite{tamelc-jsl}, a positive answer would imply that Shelah's categoricity conjecture follows from the existence of a proper class of strongly compact cardinals. Moreover, it turns out that a converse is true. This was conjectured in earlier versions of this paper, and the missing piece was proven in \cite{prime-categ-v6-toappear}:

\begin{fact}\label{building-prime-fact}
  Let $K$ be an almost fully good AEC (see Definition \ref{almost-fully-good-def}). For any $\lambda > \LS (K)^+$, $\Ksatp{\lambda}_\lambda$ has primes.
\end{fact}

Blackboxing Fact \ref{building-prime-fact}, we can give a proof of the converse of Theorem \ref{categ-nice-ap}:

\begin{thm}\label{prime-equiv}
  Let $K$ be a fully $\LS (\K)$-tame and short AEC with amalgamation. The following are equivalent:

  \begin{enumerate}
  \item\label{prime-equiv-2} $K_{\ge H_2}$ has primes and is categorical in \emph{some} $\lambda > H_2$
  \item\label{prime-equiv-1} $K$ is categorical in \emph{all} $\lambda' \ge H_2$.
  \end{enumerate}
\end{thm}
\begin{proof}
  (\ref{prime-equiv-2}) implies (\ref{prime-equiv-1}) is Theorem \ref{categ-nice-ap}. We show (\ref{prime-equiv-1}) implies (\ref{prime-equiv-2}). As in the proof of Theorem \ref{categ-nice-ap}, we assume without loss of generality that $\K$ has joint embedding and no maximal models. By Corollary \ref{fully-good-cor} (with $\kappa, \theta$ there standing for $\LS (K)^+, \lambda$ here), $K^\ast := \Ksatp{\mu}$ is almost fully good, where $\mu := \left(2^{\LS (K)}\right)^{+5}$. Now apply Fact \ref{building-prime-fact} to the AEC $K^\ast$ and use categoricity in all $\lambda' \ge H_2$.
\end{proof}
\begin{remark}
  Using the threshold improvements of \cite{downward-categ-tame-apal}, we can replace $H_2$ by $H_1$ (and allow $\lambda = H_1$ in (\ref{prime-equiv-2})) in Theorem \ref{prime-equiv}.
\end{remark}

There is still an assumption of amalgamation in Theorem \ref{categ-nice-ap}. Assuming the categoricity cardinals are sufficiently nice, this can be removed using the results of Section \ref{ap-sec}:

\begin{thm}\label{categ-primes-no-ap}
  Let $\K$ be a fully tame and short AEC with primes. If $\K$ is categorical in cardinals of arbitrarily high cofinality, then $\K$ is categorical on a tail of cardinals.
\end{thm}
\begin{proof}
  By Fact \ref{event-synt-fact}, $\K$ is eventually syntactically characterizable. By the definition of having primes, $\K$ has weak amalgamation. By Corollary \ref{tame-ap-0}, there exists $\lambda$ such that $\K_{\ge \lambda}$ has amalgamation. Now apply Theorem \ref{categ-nice-ap} to $\K_{\ge \lambda}$.
\end{proof}
\begin{remark}
  Instead of categoricity in cardinals of arbitrarily high cofinality, it suffices to assume that $\K$ is eventually syntactically characterizable and categorical in unboundedly many of cardinals.
\end{remark}

We can replace the assumption on the categoricity cardinal by large cardinals. As pointed out in the introduction (Theorem \ref{strongly-compact-thm}), a strongly compact would be enough. Here we improve this to a measurable (but assume full tameness and shortness). This only gives amalgamation below the categoricity cardinal but we can then transfer amalgamation upward using the arguments in Section \ref{ap-sec}.

\begin{thm}\label{measurable-thm}
  Let $\K$ be a fully $\LS (\K)$-tame and short AEC with primes. Let $\kappa > \LS (\K)$ be a measurable cardinal. If $\K$ is categorical in \emph{some} $\lambda > \hanf{\hanf{\kappa}}$, then $\K$ is categorical in \emph{all} $\lambda' \ge \hanf{\hanf{\kappa}}$.
\end{thm}
\begin{proof}
  By the main result of \cite{kosh362}\footnote{The result there is stated in terms of the class of models of an $L_{\kappa, \omega}$ sentence. However, Boney \cite{tamelc-jsl} has pointed out that this applies as well when $\K$ is an AEC and $\kappa > \LS (\K)$, see in particular the discussion around Theorem 7.6 there.} $\K_{[\kappa, \lambda)}$ has amalgamation (and $\K_{\ge \kappa}$ has no maximal models, using ultraproducts). Combining (the proofs of) Facts \ref{frame-construct} and \ref{shelah-vi}, there is a good $\kappa^+$-frame with underlying class $\K_{\kappa^+}$. By Theorem \ref{ap-frame-transfer}, $\K_{\ge \kappa}$ has amalgamation. Now apply Theorem \ref{categ-nice-ap} to $\K_{\ge \kappa}$.
\end{proof}

We can now prove Theorem \ref{abstract-thm-0} from the abstract.

\begin{cor}\label{abstract-thm-0-proof}
  Let $\K$ be a universal class (or just a locally pseudo-universal AEC, see Example \ref{pseudouniv-example}.(\ref{pseudouniv-1}) and Remark \ref{loc-pseudouniv-rmk}). 
  \begin{enumerate}
    \item If $\K$ is categorical in cardinals of arbitrarily high cofinality, then $\K$ is categorical on a tail of cardinals.
    \item If $\kappa > \LS (\K)$ is a measurable cardinal and $\K$ is categorical in some $\lambda > \hanf{\hanf{\kappa}}$, then $\K$ is categorical in all $\lambda' \ge \hanf{\hanf{\kappa}}$.
  \end{enumerate}
\end{cor}
\begin{proof}
  Follow the proof of Corollary \ref{categ-univ-ap} to see that the assumptions of Theorems \ref{categ-primes-no-ap} and \ref{measurable-thm} respectively are satisfied.
\end{proof}

\appendix

\section{Independence below the Hanf number}\label{good-frame-appendix}

In this appendix, we give all the results needed for the proof of Theorem \ref{get-good-frame}. We also define all the technical terms related to good frames used there. Good frames were introduced by Shelah in \cite[Chapter II]{shelahaecbook} but we use the notation and definitions in \cite{indep-aec-apal} (we also extensively use its results). The reader is invited to consult this paper for more motivation and background on the concepts used here.

The first definition is that of a \emph{global} forking-like notion:

\begin{defin}[Definition 8.1 in \cite{indep-aec-apal}]\label{fully-good-def}
  $\is = (K, \nf)$ is a \emph{fully good independence relation} if:

  \begin{enumerate}
    \item $K$ is an AEC with $K_{<\LS (K)} = \emptyset$ and $K \neq \emptyset$.
    \item $K$ has amalgamation, joint embedding, and no maximal models.
    \item $K$ is stable in all cardinals.
    \item $\is$ is a $(<\infty, \ge \LS (K))$-independence relation (see \cite[Definition 3.6]{indep-aec-apal}). That is, $\nf$ is a relation on quadruples $(M, A, B, N)$ with $M \lea N$ and $A, B \subseteq |N|$ satisfying invariance, monotonicity, and normality. We write $\nfs{M}{A}{B}{N}$ instead of $\nf (M, A, B, N)$, and we also say $\gtp (\ba / B; N)$ does not fork over $M$ for $\nfs{M}{\ran{\ba}}{B}{N}$.
    \item $\is$ has base monotonicity, disjointness ($\nfs{M}{A}{B}{N}$ implies $A \cap B \subseteq |M|$), symmetry, uniqueness, extension, and the local character properties:
      \begin{enumerate}
        \item If $p \in \gS^\alpha (M)$, there exists $M_0 \lea M$ with $\|M_0\| \le |\alpha| + \LS (K)$ such that $p$ does not fork over $M_0$.
        \item If $\seq{M_i : i \le \delta}$ is increasing continuous, $p \in \gS^\alpha (M_\delta)$ and $\cf{\delta} > \alpha$, then there exists $i < \delta$ such that $p$ does not fork over $M_i$.
      \end{enumerate}
    \item $\is$ has the left and right $(\le \LS (K))$-witness properties: $\nfs{M}{A}{B}{N}$ if and only if for all $A_0 \subseteq A$ and $B_0 \subseteq B$ with $|A_0| + |B_0| \le \LS (K)$, we have that $\nfs{M}{A_0}{B_0}{N}$.
    \item $\is$ has full model continuity: if for $\ell < 4$, $\seq{M_i^\ell : i \le \delta}$ are increasing continuous such that for all $i < \delta$, $M_i^0 \lea M_i^\ell \lea M_i^3$ for $\ell = 1,2$ and $\nfs{M_i^0}{M_i^1}{M_i^2}{M_i^3}$, then $\nfs{M_\delta^0}{M_\delta^1}{M_\delta^2}{M_\delta^3}$.
  \end{enumerate}

  We say that $\is$ is \emph{good} if it has all the properties above except full model continuity. We say that $K$ is \emph{[fully] good} if there exists $\nf$ such that $(K, \nf)$ is [fully] good.
\end{defin}

We will use the following variation:

\begin{defin}\label{almost-fully-good-def}
  $\is = (K, \nf)$ is \emph{almost fully good} if it satisfies Definition \ref{fully-good-def} except that only the following types are required to have a nonforking extension:

  \begin{enumerate}
    \item Types that do not fork over \emph{saturated} models.
    \item Type that do not fork over models of size $\LS (K)$.
    \item Types of length at most $\LS (K)$.
  \end{enumerate}

  As before, we say that $K$ is \emph{almost fully good} if there exists $\nf$ such that $(K, \nf)$ is almost fully good. If we drop ``fully'' we mean that full model continuity need not hold.
\end{defin}

In this terminology, we have:

\begin{fact}[Theorem 15.1.(3) in \cite{indep-aec-apal}]\label{good-frame-succ}
  Let $K$ be a fully $(<\kappa)$-tame and short AEC with amalgamation. 

  If $\kappa = \beth_{\kappa} > \LS (K)$, and $K$ is categorical in a $\mu > \lambda_0 := \left(2^\kappa\right)^{+5}$, then $K_{\ge \lambda}$ is almost fully good, where we have set $\lambda := \min (\mu, \hanf{\lambda_0})$.
\end{fact}

A localization of fully good independence relation are Shelah's good $\lambda$-frames. Roughly speaking, we simply require the types to have length one and the models to have a fixed size $\lambda$. We only give the definition of a \emph{type-full} good $\lambda$-frame here, since this is the one that we can build here. In \cite[Section II.2]{shelahaecbook}, Shelah has a more general definition where he only requires a dense class of basic types to satisfy the properties of forking: this is also what we call a good $\lambda$-frame (without the ``type-full'') in this paper, e.g.\ in Theorem \ref{ap-frame-transfer}. We use the definition in \cite[Definition 8.1.(2)]{indep-aec-apal} and refer to Remark 3.5 there for why this is equivalent (in the type-full case) to Shelah's definition in \cite[Section II.2]{shelahaecbook}. 

\begin{defin}
  $\s = (K_{\s}, \nf)$ is a \emph{type-full good $\lambda$-frame} if:

  \begin{enumerate}
    \item There exists an AEC $K$ with $\lambda = \LS (K)$, $K_\lambda = K_{\s}$. Below, we require that all the models be in $K_{\s}$.
    \item $K_{\s} \neq \emptyset$.
    \item $K_{\s}$ has amalgamation, joint embedding, and no maximal models.
    \item $K_{\s}$ is stable in $\lambda$.
    \item $\nf$ is a relation on quadruples $(M_0, a, M, N)$ with $M_0 \lea M \lea N$ and $a \in |N|$ satisfying invariance, monotonicity, and normality. As before, we write $\nfs{M_0}{a}{M}{N}$ instead of $\nf (M_0, a, M, N)$, and we also say $\gtp (a / M; N)$ does not fork over $M_0$ for $\nfs{M_0}{a}{M}{N}$.
    \item $\s$ has base monotonicity, disjointness, full symmetry (if $\nfs{M_0}{a}{M}{N}$, $b \in |M|$, then there exists $N' \gea N$ and $M_0' \gea M_0$ with $M_0' \lea N'$, $a \in |M_0'|$, and $\nfs{M_0}{b}{M_0'}{N'}$), uniqueness, extension, and the local character property: If $\seq{M_i : i \le \delta}$ is increasing continuous, $p \in \gS (M_\delta)$, then there exists $i < \delta$ such that $p$ does not fork over $M_i$.
  \end{enumerate}

  We define similarly ``type-full good $(\ge \lambda)$-frame'', where we allow the models in $K_{\s}$ to have sizes in $K_{\ge \lambda}$ (but still work with types of length one).
\end{defin}
\begin{notation}
When $\is = (K, \nf)$ is an almost good independence relation and $\lambda \ge \LS (K)$, we write $\pre (\is^{\le 1}) \rest K_\lambda$ for the type-full good $\lambda$-frame obtained by restricting $\nf$ to types of length one and models in $K_\lambda$. Similarly for $\pre (\is^{\le 1}) \rest K_{\ge \lambda}$.
\end{notation}

Assuming tameness and amalgamation, good frames can be built from a superstability-like condition (the superstability condition already appears implicitly in \cite{shvi635} and is developed further in \cite{vandierennomax, nomaxerrata, gvv-mlq, ss-tame-jsl, indep-aec-apal, bv-sat-v3, vandieren-symmetry-apal, vv-symmetry-transfer-v4}). The construction of a good frame appears implicitly already in \cite{ss-tame-jsl}:

\begin{defin}[Superstability, see Definition 10.1 in \cite{indep-aec-apal}]\label{ss-def} \
  \begin{enumerate}
    \item For $M, N \in K$, say $M \ltu N$ ($N$ is \emph{universal over $M$}) if and only if $M \lta N$ and whenever we have $M' \gea M$ such that $\|M'\| \le \|N\|$, then there exists $f: M' \xrightarrow[M]{} N$. Say $M \leu N$ if and only if $M = N$ or $M \ltu N$.
    \item $p \in \gS (N)$ \emph{$\mu$-splits over $M$} if $M \lea N$, $M \in K_\mu$, and there exists $N_1, N_2 \in K_\mu$ with $M \lea N_\ell \lea N$, $\ell = 1,2$, and an isomorphism $f : N_1 \cong_M N_2$, such that $f (p \rest N_1) \neq p \rest N_2$.
    \item $K$ is \emph{$\mu$-superstable} if:
      \begin{enumerate}
      \item $\LS (K) \le \mu$.
      \item\label{cond-two-mods} There exists $M \in K_\mu$ such that for any $M' \in K_\mu$ there is $f: M' \rightarrow M$ with $f[M'] \ltu M$.
      \item If $\seq{M_i : i < \delta}$ is increasing in $K_\mu$ such that $i < \delta$ implies $M_i \ltu M_{i + 1}$ and $p \in \gS (\bigcup_{i < \delta} M_i)$, then there exists $i < \delta$ such that $p$ does not $\mu$-split over $M_i$.
      \end{enumerate}
  \end{enumerate}
\end{defin}

\begin{defin}
  For $\lambda$ a cardinal, let $\Ksatp{\lambda}$ be the class of $\lambda$-saturated models in $K_{\ge \lambda}$.
\end{defin}

\begin{fact}\label{frame-construct}
  Assume $K$ is $\mu$-superstable, $\mu$-tame, and has amalgamation. Then:
  
  \begin{enumerate}
    \item\cite[Proposition 10.10]{indep-aec-apal}] $K$ is $\mu'$-superstable for all $\mu' \ge \mu$. In particular, $K_{\ge \mu}$ has joint embedding, no maximal models, and is stable in all cardinals.
    \item \cite[Corollary 6.10]{vv-symmetry-transfer-v4} If $\lambda > \mu$, then $\Ksatp{\lambda}$ is an AEC with $\LS (\Ksatp{\lambda}) = \lambda$.
    \item (\cite[Corollary 6.14]{vv-symmetry-transfer-v4} and \cite[Corollary 6.10]{tame-frames-revisited-v6-toappear}). For any $\lambda > \mu$, there exists a type-full good $(\ge \lambda)$-frame with underlying AEC $\Ksatp{\lambda}$.
  \end{enumerate}
\end{fact}

From the analysis of Shelah and Villaveces in \cite[Theorem 2.2.1]{shvi635}, we obtain that superstability follows from categoricity (if the cofinality of the categoricity cardinal is high-enough, this appears as \cite[Lemma 6.3]{sh394}). The version that we state here assumes amalgamation instead of GCH and appears in \cite{shvi-notes-v3-toappear}.

\begin{fact}\label{shelah-vi}
  Assume $K$ has amalgamation and no maximal models. If $K$ is categorical in a $\theta > \LS (K)$, then $K$ is $\LS (K)$-superstable.
\end{fact} 

\begin{cor}\label{frame-exist}
  Assume $K$ is $\LS (K)$-tame and has amalgamation and no maximal models. If $K$ is categorical in a $\theta > \LS (K)$, then there exists a type-full good $(\ge \LS (K)^+)$-frame with underlying class $\Ksatp{\theta^+}$.
\end{cor}
\begin{proof}
  Combine Facts \ref{shelah-vi}, and \ref{frame-construct}.
\end{proof}

It remains to see how to build a fully good (i.e.\ global) independence relation from just a local good frame. This is done using shortness, together with a property Shelah calls \emph{successfulness} (we do not give the exact definition of uniqueness triple, the relation $\lea_{\lambda^+}^{\text{NF}}$, or the successor frame, as we have no use for it).

\begin{defin}[Definition III.1.1 in \cite{shelahaecbook}]\label{successful-def}
  Let $\s$ be a type-full good $\lambda$-frame.

  \begin{enumerate}
    \item $\s$ is \emph{weakly successful} if for any $M \in K_\lambda$ and any nonalgebraic $p \in \gS (M)$, there exists $N \gea M$ and $a \in |N|$ such that $p = \gtp (a / M; N)$ and $(a, M, N)$ is a uniqueness triple (see \cite[Definition II.5.3]{shelahaecbook}).
    \item $\s$ is \emph{successful} if in addition the class $(\Ksatp{\lambda^+}_{\lambda^+}, \lea_{\lambda^+}^{\text{NF}})$ (see \cite[Definition 10.1.1]{jrsh875}) is an AEC.
    \item \cite[Definition III.1.12]{shelahaecbook} $\s$ is \emph{$\omega$-successful} if for all $n < \omega$, the $n$th successor frame $\s^{+n}$ (see \cite[Definition III.1.12]{shelahaecbook}) is a type-full successful good $\lambda$-frame.
  \end{enumerate}
\end{defin}

We can obtain an $\omega$-successful frame using existence of a sufficiently well-behaved global independence relation:

\begin{fact}[Theorem 11.21 in \cite{indep-aec-apal}]\label{weakly-succ-build}
  Assume $\is$ is a $(<\infty, \ge \LS (K))$-independence relation on $K$ and $\lambda > \LS (\K)$ is a cardinal such that\footnote{In \cite{indep-aec-apal}, it is also assumed that $\Ksatp{\lambda^{+n}}$ is an AEC for all $n < \omega$. However Fact \ref{frame-construct} shows that this follows from the rest.}:

  \begin{enumerate}
    \item $\s := \pre (\is^{\le 1})$ is a type-full good $(\ge \LS (K))$-frame.
    \item $\is$ has base monotonicity, uniqueness for types over models, and the left and right $(\le \LS (K))$-witness properties.
    \item $\is$ has the following local character property: for every $n < \omega$, if $\mu := \lambda^{+(n + 1)}$, then for every increasing continuous $\seq{M_i : i \le \mu}$ and every $p \in \gS^{< \mu} (M_\mu)$, there exists $i < \mu$ such that $p$ does not fork (in the sense of $\is$) over $M_i$.
  \end{enumerate}

  Then $\s \rest \Ksatp{\lambda}_{\lambda}$ (the restriction of $\s$ to the class $\Ksatp{\lambda}_{\lambda}$) is an $\omega$-successful type-full good $\lambda$-frame.
\end{fact}

In \cite{indep-aec-apal}, we used $(<\kappa)$-satisfiability as $\is$ above. The downside is that we used that $\kappa = \beth_\kappa > \LS (K)$. Now we show we can use an independence relation induced by $\mu$-nonsplitting instead of $(<\kappa)$-satisfiability. We need one more fact:

\begin{fact}\label{ns-lc}
  Assume $K$ has amalgamation and is stable in $\mu \ge \LS (K)$. Let $M \in K_{\ge \mu}$ and let $p \in \gS^{<\kappa} (M)$. If $\mu = \mu^{<\kappa}$, then there exists $M_0 \in K_\mu$ with $M_0 \lea M$ such that $p$ does not $\mu$-split over $M_0$.
\end{fact}
\begin{proof}
  By \cite[Claim 3.3]{sh394} (or \cite[Fact 4.6]{tamenessone}), it is enough to show that $|\gS^{<\kappa} (N)| = \mu$ for every $N \in K_\mu$. This holds by stability in $\mu$ and \cite[Theorem 3.1]{longtypes-ndjfml}.
\end{proof}

\begin{lem}\label{omega-succ-constr}
  Let $\K$ be an AEC with amalgamation. Assume that $\K$ is fully $(<\kappa)$-tame and short with $\kappa \le \LS (\K)^+$. Assume further that $\K$ is $\LS (\K)$-superstable. Let $\lambda > \LS (\K)$ be such that $\lambda = \lambda^{<\kappa}$. Then there exists an $\omega$-successful good $\lambda^+$-frame with underlying class $\Ksatp{\lambda^+}_{\lambda^+}$.
\end{lem}
\begin{proof}[Proof sketch]
  Define a $(<\infty, \ge \lambda)$ independence relation $\is = (K_{\is}, \nf)$ as follows:

  \begin{itemize}
  \item $K_{\is} = \Ksatp{\lambda}$.
  \item $p \in \gS^{\alpha} (M)$ does not fork (in the sense of $\is$) over $M_0 \lea M$ if and only if:

    \begin{itemize}
        \item $M_0, M \in \Ksatp{\lambda}$.
        \item For every $I \subseteq \alpha$ with $|I| < \kappa$, there exists $M_0' \lea M_0$ in $K_\mu$ such that $p^I$ does not $\mu$-split over $M_0'$.
    \end{itemize}
  \end{itemize}  

  We claim that $\is$ satisfies the hypotheses of Fact \ref{weakly-succ-build} (where $K$ there is $K_{\ge \lambda}$ here and $\lambda$ there is $\lambda^+$ here). By Fact \ref{ns-lc} and superstability, we have that $\is$ induces a $\LS (K)$-generator for a weakly good $(<\kappa)$-independence relation (in the sense of \cite[Definition 7.3]{indep-aec-apal}), as well as a $\LS (K)$-generator for a good $(\le 1)$-independence relation (see \cite[Definition 8.5]{indep-aec-apal}). It follows from \cite[Theorems 7.5, 8.9]{indep-aec-apal} that $\s := \pre (\is^{\le 1})$ is a type-full good $(\ge \lambda)$-frame and $\is^{<\kappa}$ (the restriction of $\is$ to types of length less than $\kappa$) has base monotonicity, uniqueness for types over models, transitivity, and so that any type does not fork over a model of size $\LS (K)$.

  Now, it is easy to see using shortness that $\is$ also has uniqueness for types over models. By definition, it also has base monotonicity, transitivity, and the left $(<\kappa)$-witness property. Now from transitivity and the local character property mentioned in the previous paragraph, we get (\cite[Proposition 4.3.(6)]{indep-aec-apal}) that $\is$ has the right $(\le \LS (K))$-witness property. Thus all the hypotheses of Fact \ref{weakly-succ-build} are satisfied, so $\s$ is $\omega$-successful.
\end{proof}

From an $\omega$-successful good $\lambda$-frame, we obtain the desired global independence relation:

\begin{fact}\label{fully-good-constr}
  Let $\s = (K_\lambda, \nf)$ be an $\omega$-successful good $\lambda$-frame which is categorical in $\lambda$. If $K$ is fully $(<\cf{\lambda})$-tame and short and has amalgamation, then $\Ksatp{\lambda^{+3}}$ is almost fully good.
\end{fact}
\begin{proof}
  By \cite[Theorems 12.16, 13.6, 14.15]{indep-aec-apal}
\end{proof}
\begin{cor}\label{fully-good-cor}
  Assume that $K$ has amalgamation, no maximal models, and is fully $(<\kappa)$-tame and short, with $\kappa \le \LS (\K)^+$ a regular cardinal. If $K$ is categorical in a $\theta > \LS (K)$, then $\Ksatp{\lambda}$ is almost fully good, where $\lambda := \left(\LS (\K)^{<\kappa}\right)^{+5}$.
\end{cor}
\begin{proof}
  By Fact \ref{shelah-vi}, $K$ is $\LS (\K)$-superstable. Let $\mu := \left(\LS (\K)^{<\kappa}\right)^{+}$. By Fact \ref{omega-succ-constr} (with $\lambda$ there standing for $\mu$ here), there is an $\omega$-successful good $\mu^+$-frame with underlying class $\Ksatp{\mu^+}_{\mu^+}$. By Fact \ref{fully-good-constr} (with $\lambda$ there standing for $\mu^+$ here), $\Ksatp{\mu^{+4}}$ is almost fully good.
\end{proof}

Note for future reference that in almost good AECs, uniqueness triples have an easier definition.

\begin{defin}\label{dom-def}
  Let $\is = (K, \nf)$ be an almost good independence relation. $(a, M, N)$ is a \emph{domination triple} if $M \lea N$, $a \in |N| \backslash |M|$, and for any $N' \gea N$ and any $B \subseteq |N'|$, if $\nfs{M}{a}{B}{N'}$, then $\nfs{M}{N}{B}{N'}$.
\end{defin}

\begin{fact}[Lemma 11.7 in \cite{indep-aec-apal}]\label{uq-domin}
  Let $\is = (K, \nf)$ be an almost good independence relation. Let $\mu \ge \LS (K)$ and let $\s := \pre (\is^{\le 1}) \rest K_\mu$.

  For $M, N \in K_\mu$, $(a, M, N)$ is a domination triple if and only it is a uniqueness triple in $\s$.
\end{fact}

We continue the proof of Theorem \ref{get-good-frame} by showing that the frame induced by an almost good independence relation is $\goodp$, a technical property of frames:

\begin{defin}[Definition III.1.3 in \cite{shelahaecbook}]
  Let $\s = (K_\lambda, \nf)$ be a type-full good $\lambda$-frame. $\s$ is $\goodp$ if the following is \emph{impossible}: There exists increasing continuous chains $\seq{M_i : i \le \lambda^+}$, $\seq{N_i : i \le \lambda^+}$, a type $p^\ast \in \gS (M_0)$, and a sequence $\seq{a_i : i < \lambda^+}$ such that for all $i < \lambda^+$:

  \begin{enumerate}
    \item $M_{\lambda^+}$ is $\lambda^+$-saturated.
    \item $M_i \lea N_i$ and they are both in $K_\lambda$.
    \item $a_{i + 1} \in |M_{i + 2}|$.
    \item $\gtp (a_{i + 1} / M_{i + 1}; M_{i + 2})$ is a nonforking extension of $p^\ast$.
    \item $\gtp (a_{i + 1} / N_0; N_{i + 2})$ forks over $M_0$.
  \end{enumerate}
\end{defin}

\begin{prop}\label{indep-goodp}
  If $\is = (K, \nf)$ is an almost good independence relation, then $\pre (\is^{\le 1}) \rest K_{\LS (K)}$ is $\goodp$.
\end{prop}
\begin{proof}
  Suppose $\seq{M_i : i \le \lambda^+}$, $\seq{N_i : i \le \lambda^+}$, $\seq{a_i : i < \lambda^+}$, and $p^\ast$ witness the failure of being $\goodp$. By local character, there exists $i < \lambda^+$ such that $\nfs{M_i}{N_0}{M_{\lambda^+}}{N_{\lambda^+}}$. By symmetry and monotonicity, we must have that $\nfs{M_i}{a_{i + 1}}{N_0}{N_{\lambda^+}}$, i.e.\ $\gtp (a_{i + 1} / N_0; N_{i + 2})$ does not fork over $M_i$. By transitivity and base monotonicity, $\gtp (a_{i + 1} / N_0; N_{i + 2})$ does not fork over $M_0$, contradiction.
\end{proof}

\begin{cor}\label{get-good-frame-cor}
  Assume $\is = (K, \nf)$ is an almost good independence relation. Let $\lambda > \LS (K)$ and let $\s := \pre (\is^{\le 1}) \rest \Ksatp{\lambda}$. Then $\s$ is $\omega$-successful and $\goodp$.
\end{cor}
\begin{proof}
  By Fact \ref{weakly-succ-build} and Proposition \ref{indep-goodp} (applied to the restriction of $\is$ to $\lambda$-saturated models).
\end{proof}

\begin{remark}
  In Section \ref{categ-transfer-sec}, we only need a (type-full) successful $\goodp$ frame. Moreover Shelah proves in \cite[Claim III.1.9]{shelahaecbook} that if $\s$ is successful, then the successor frame $\s^+$ is $\goodp$, so why do we bother building an almost good independence relation? The reason is that we want a successful $\goodp$ $\lambda$-frame when $\lambda$ is a \emph{limit} cardinal. Then if $\K$ is categorical in $\lambda$ and has primes, the frame will have primes (no need to restrict to saturated models, where it is not clear whether primes exist even if the original class has primes), so the hypotheses of Theorem \ref{categ-transfer-frame} will be satisfied
\end{remark}

\section{Frames that are not weakly uni-dimensional}\label{proof-appendix}

In this appendix, we give a proof of Fact \ref{not-unidim-frame}. We work with the following hypotheses:

\begin{hypothesis} \
  \begin{enumerate}
    \item $\s = (K_\lambda, \nf)$ is a type-full successful $\goodp$ $\lambda$-frame.
    \item $K_\lambda$ has primes.
    \item $K$ is categorical in $\lambda$.
  \end{enumerate}
\end{hypothesis}

We will use the orthogonality calculus developed in \cite[Chapter III]{shelahaecbook}.

\begin{defin}[Definition III.6.2 in \cite{shelahaecbook}] \
  \begin{enumerate}
    \item Let $M \in K_\lambda$ and let $p, q \in \gS (M)$ be nonalgebraic. We say that $p$ and $q$ are \emph{weakly orthogonal} if whenever $(a, M, N)$ is a uniqueness triple with $\gtp (a / M; N) = q$, then $p$ has a unique extension to $\gS (N)$. We say that $p$ and $q$ are \emph{orthogonal}, written $p \perp q$ if for every $N \gea M$, the nonforking extensions to $N$ $p'$, $q'$ of $p$ and $q$ respectively are weakly orthogonal.
    \item Let $M_\ell \in K_\lambda$ and $p_\ell \in \gS (M_\ell)$ be nonalgebraic, $\ell = 1,2$. We say that $p_1$ and $p_2$ are \emph{orthogonal} if there exists $N \gea M_\ell$ such that the nonforking extensions to $N$ $p_1'$, $p_2'$ of $p_1$ and $p_2$ respectively are orthogonal.
  \end{enumerate}
\end{defin}

\begin{fact}[Claims III.6.7, III.6.8 in \cite{shelahaecbook}]\label{orthog-facts}
  Let $M \in K_\lambda$ and $p, q \in \gS (M)$ be nonalgebraic.
  \begin{enumerate}
    \item \cite[Claim III.6.3]{shelahaecbook} $p$ is weakly orthogonal to $q$ if and only if there exists a uniqueness triple $(a, M, N)$ such that $\gtp (a / M; N) = q$ and $p$ has a unique extension to $\gS (N)$.
    \item \cite[Claim III.6.7.2]{shelahaecbook} $p \perp q$ if and only if $q \perp p$.
    \item \cite[Claim III.6.8.5]{shelahaecbook} $p$ and $q$ are orthogonal if and only if they are weakly orthogonal.
  \end{enumerate}
\end{fact}

We will also use the following without comments:

\begin{fact}[Claim III.3.7 in \cite{shelahaecbook}]\label{uq-prime}
  If $(a, M, N)$ is a prime triple, then it is a uniqueness triple.
\end{fact}

Some orthogonality calculus gives us a useful description of the types in $K_{\neg^\ast p}$ (recall Definition \ref{kneg-def}).

\begin{lem}\label{kneg-orth}
  Let $M \in K_\lambda$ and let $p \in \gS (M)$ be nonalgebraic. Let $N \in K_{\neg^\ast p}$ be of size $\lambda$ such that the map $a \mapsto c_a^N$ is the identity (so $M \lea N \rest L(K)$). For any $N_0 \lea N \rest L(K)$ with $M \lea N_0$ and any $q \in \gS (N_0; N)$, $p \perp q$.
\end{lem}
\begin{proof}
  Let $p'$ be the nonforking extension of $p$ to $N_0$. By Fact \ref{orthog-facts}, it is enough to show that $p'$ is weakly orthogonal to $q$. Let $(a, N_0, N')$ be a prime triple such that $\gtp (a / N_0; N') = q$ and $N' \lea N$ (exists since we are assuming that $K_\lambda$ has primes). Then since $p$ has a unique extension to $N$ it has a unique extension to $N'$, which must be the nonforking extension so $p'$ also has a unique extension to $N'$. By Fact \ref{uq-prime}, ($a, N_0, N')$ is a uniqueness triple and by Fact \ref{orthog-facts} again, this suffices to conclude that $p'$ and $q$ are weakly orthogonal.
\end{proof}

The next lemma justifies the ``uni-dimensional'' terminology: if the class is \emph{not} uni-dimensional, then there are two orthogonal types.

\begin{lem}\label{p-q-lem}
  If $K_\lambda$ is not weakly uni-dimensional, there exists $M \in K_\lambda$ and types $p, q \in \gS (M)$ such that $p \perp q$.
\end{lem}
\begin{proof}
  Assume $K_\lambda$ is not weakly uni-dimensional. This means that there exists $M \lta M_\ell$, $\ell = 1,2$, all in $K_\lambda$ such that for any $c \in |M_2| \backslash |M|$, $\gtp (c / M; M_2)$ has a unique extension to $\gS (M_1)$. Pick any $c \in |M_2| \backslash |M|$ and let $p := \gtp (c / M; M_2)$. Then there is a natural expansion of $M_1$ to $K_{\neg^\ast p}$. So pick any $d \in |M_1| \backslash |M|$ and let $q := \gtp (d / M; M_1)$. By Lemma \ref{kneg-orth}, $p \perp q$, as desired.
\end{proof}

We can now prove Fact \ref{not-unidim-frame}. We restate it here for convenience:

\begin{fact}\label{not-unidim-frame-1}
  If $K_\lambda$ is not weakly uni-dimensional, then there exists $M \in K_\lambda$ and $p \in \gS (M)$ such that $\s \rest K_{\neg^\ast p}$ (the restriction of $\s$ to the models in $K_{\neg^\ast p}$) is a type-full good $\lambda$-frame.
\end{fact}
\begin{proof}
  Assume $K_\lambda$ is not weakly uni-dimensional. By Lemma \ref{p-q-lem}, there exists $M \in K_\lambda$ and types $p, q \in \gS (M)$ such that $p \perp q$.

  Let $\s_{\neg^\ast p} := \s \rest K_{\neg^\ast p}$. We check that it is a type-full good $\lambda$-frame. For ease of notation, we identify a model $N \in K_{\neg^\ast p}$ and its reduct to $K$. For $N \gea M$, we write $p_N$ for the nonforking extension of $p$ to $\gS (N)$, and similarly for $q_N$. 

  \begin{itemize}
    \item $K_{\neg^\ast p}$ is not empty, since (the natural expansion of) $M$ is in it.
    \item $(K_{\neg^\ast p})_\lambda$ is an AEC in $\lambda$ (that is, its models of size $\lambda$ behave like an AEC, see \cite[Definition II.1.18]{shelahaecbook}) by the proof of Proposition \ref{kneg-aec}.
    \item Forking has many of the usual properties: monotonicity, invariance, disjointness, local character, continuity, and transitivity all trivially follow from the definition of $K_{\neg^\ast p}$.
    \item Forking has the uniqueness property: Let $N \in K_{\neg^\ast p}$ have size $\lambda$. Without loss of generality $M \lea N$. Let $N' \gea N$ be in $K_{\neg p}$ of size $\lambda$ and let $r_1, r_2 \in \gS (N')$ be nonforking over $N$ and such that $r_1 \rest N = r_2 \rest N$. Say $r_\ell = \gtp (a_\ell / N'; N_\ell)$. Now in $K$, $r_1 = r_2$, and since $K_\lambda$ has primes, the equality is witnessed by an embedding $f: M_1 \xrightarrow[N]{} N_2$, with $M_1 \lea N_1$. Since $N_1 \in K_{\neg^\ast p}$, $M_1 \in K_{\neg^\ast p}$, and so $r_1 = r_2$ also in $K_{\neg^\ast p}$ (this is similar to the proof of Proposition \ref{weakly-prime-nice}).
    \item Forking has the extension property. Let $N \in K_{\neg^\ast p}$ have size $\lambda$. Without loss of generality, $M \lea N$. Let $r \in \gS (N)$ be nonalgebraic and let $N' \gea N$ be in $K_{\neg^\ast p}$ of size $\lambda$. Let $r' \in \gS (N')$ be the nonforking extension of $r$ to $N'$ (in $K$). Let $(a, N', N'')$ be a prime triple such that $\gtp (a / N'; N'') = r'$. By Lemma \ref{kneg-orth}, $r \perp p$. Thus $r'$ is weakly orthogonal to $p_{N'}$ and hence $p_{N''}$ is the unique extension of $p_{N'}$ to $N''$. Now if $p'$ is an extension of $p$ to $\gS (N'')$, then $p' \rest N' = p_{N'}$ as $N' \in K_{\neg^\ast p}$, so $p' = p_{N''}$ by the previous sentence. This shows that $N'' \in K_{\neg^\ast p}$, so as $r' \in \gS (N'; N'')$, $r'$ is a Galois type in $K_{\neg^\ast p}$, as desired.
    \item $K_{\neg^\ast p}$ has $\lambda$-amalgamation: because $(K_{\neg^\ast p})_\lambda$ has the type extension property and weak $\lambda$-amalgamation (as $K_\lambda$, and hence $(K_{\neg^\ast p})_\lambda$, has primes, see the proof of Proposition \ref{weakly-prime-nice}), thus one can apply Theorem \ref{thm-closure}.
    \item $K_{\neg^\ast p}$ has $\lambda$-joint embedding: since any model contains a copy of $M$, this is a consequence of $\lambda$-amalgamation over $M$.
    \item $K_{\neg^\ast p}$ is stable in $\lambda$: because $K_{\neg^\ast p}$ has ``fewer'' Galois types than $K$, and $K$ is stable in $\lambda$.
    \item $(K_{\neg^\ast p})_\lambda$ has no maximal models: This is where we use the negation of weakly uni-dimensional. Let $N \in K_{\neg^\ast p}$ be of size $\lambda$ and without loss of generality assume $M \lea N$. Recall from above that there is a nonalgebraic type $q \in \gS (M)$ such that $p \perp q$. Let $q_N$ be the nonforking extension of $q$ to $N$ and let $(a, N, N')$ be a prime triple such that $q = \gtp (a / N; N')$. As in the proof of the extension property, $N' \in K_{\neg^\ast p}$. Moreover as $a \in |N'| \backslash |N|$, $N \lta N'$, as needed.
    \item $\s_{\neg^\ast p}$ is type-full: because $\s$ is.
    \item $\s_{\neg^\ast p}$ has full symmetry: Assume $\nfs{N_0}{a}{N_1}{N}$, for $N_0, N_1, N \in K_{\neg^\ast p}$, $M \lea N_0 \lea N_1 \lea N$, and $a \in |N|$. Let $b \in |N_1|$. Without loss of generality, $a \notin |N_1|$ (if $a \in |N_1|$, then $a \in |N_0|$ by disjointness and as $\nfs{N_0}{b}{N_0}{N}$, $N_0$ and $N$ witness the full symmetry). By full symmetry in $\s$, there exists $N_0', N' \in K$ such that $N \lea N'$, $N_0 \lea N_0' \lea N'$, and $\nfs{N_0}{b}{N_0'}{N'}$ (note that the first use of $\nf$ was in $\s_{\neg^\ast p}$ and the second in $\s$, but since the first is just the restriction of the first to models in $K_{\neg^\ast p}$, we do not make the difference). Now let $N_0''$ be such that $N_0 \lea N_0'' \lea N_0'$ and $(a, N_0, N_0'')$ is a prime triple. Since $r = \gtp (a/ N_0; N_0'') = \gtp (a / N_0; N)$ is orthogonal to $p$ (by Lemma \ref{kneg-orth}), we have that $N_0'' \in K_{\neg^\ast p}$. By monotonicity, $\nfs{N_0}{b}{N_0''}{N'}$. Now let $(b, N_0'', N'')$ be a prime triple with $N'' \lea N'$. By the same argument as before, $N'' \in K_{\neg^\ast p}$ and by monotonicity, $\nfs{N_0}{b}{N_0''}{N''}$. Since all the models are in $K_{\neg^\ast p}$, this shows that the nonforking happens in $\s_{\neg^\ast p}$, as needed.
  \end{itemize}
  
  We have checked all the properties and therefore $\s_{\neg^\ast p}$ is a type-full good $\lambda$-frame.
\end{proof}

\section{Independence in universal classes}\label{indep-sec}

We investigate the properties of independence in universal classes (more generally in AECs admitting intersections). Recall that \cite[Theorem 15.6]{indep-aec-apal} showed that a fully tame and short AEC with amalgamation categorical in unboundedly many cardinals eventually admits a well-behaved independence notion. We want to specialize this result to AECs admitting intersections and prove more properties of forking there. Here, we prove that the independence relation satisfies the axioms of \cite{bgkv-apal} (partially answering Question 7.1 there). Moreover it has a finite character property (Theorem \ref{fin-char-thm}) and can be extended to an independence relation over sets (Theorem \ref{set-base-thm}). A simple corollary is the disjoint amalgamation property (Corollary \ref{dap-cor}).

While none of the results are used in this paper, we believe they shed further light on how the existence a closure operator helps in the structural analysis of an AEC. Since several classes of interests to algebraists admit intersections, we believe the existence of a well-behaved independence notion there is likely to have further applications.

By Fact \ref{good-frame-succ} or Corollary \ref{fully-good-cor}, it is reasonable to assume:

\begin{hypothesis} \
  \begin{enumerate}
    \item $K$ locally admits intersections.
    \item $\is = (K, \nf)$ is an almost fully good independence relation (see Definition \ref{fully-good-def}).
  \end{enumerate}
\end{hypothesis}

Our goal is to prove that $\is$ is actually fully good, i.e.\ extension holds. Note that if we knew that $K$ was categorical above the Löwenheim-Skolem-Tarski number, we could use the categoricity transfer of Section \ref{categ-transfer-sec}. However here we do not make any categoricity assumption and our approach is easier: we study how the closure operator interacts with independence. The key lemma is:

\begin{lem}\label{cl-indep}
  If $\nfs{M_0}{A}{B}{N}$, then $\nfs{M_0}{\cl^N (A)}{\cl^N (B)}{N}$.
\end{lem}
\begin{proof}
  By normality, without loss of generality $|M_0| \subseteq A, B$. Using symmetry, it is enough to show that $\nfs{M_0}{A}{\cl^N (B)}{N}$. By the witness property and finite character of the closure operator, we can assume without loss of generality that $|A| \le \LS (K)$. Therefore by extension there exists $N' \gea N$ and $M \gea M_0$ such that $M \lea N'$, $M$ contains $B$, and $\nfs{M_0}{A}{M}{N'}$. By definition, $\cl^N (B) = \cl^{N'} (B)$ is contained in $M$, so $\nfs{M_0}{A}{\cl^N (B)}{N'}$, so $\nfs{M_0}{A}{\cl^N (B)}{N}$.
\end{proof}

An abstract way of stating Lemma \ref{cl-indep} is via domination triples (recall Definition \ref{dom-def}). 

\begin{lem}\label{uq-triple-mu}
  Let $M \lea N$ and let $a \in |N| \backslash |M|$. Then $(a, M, \cl^N (\{a\} \cup |M|))$ is a domination triple.
\end{lem}
\begin{proof}
  Directly from Lemma \ref{cl-indep}.
\end{proof}

In our framework, domination triples are the same as the uniqueness triples of \cite[Definition II.5.3]{shelahaecbook} by Fact \ref{uq-domin}, thus we get:

\begin{thm}\label{fully-good-ir}
  $\is$ has extension. Hence it is a fully good independence relation.
\end{thm}
\begin{proof}
  Let $\mu \ge \LS (K)$ and let $\s := \pre (\is^{\le 1} \rest K_\mu)$. By Lemma \ref{uq-triple-mu} and Fact \ref{uq-domin} $\s$ has the so-called existence property for uniqueness triples (see \cite[Definition II.5.3]{shelahaecbook}). By Section II.6 of \cite{shelahaecbook} (and the results of section 12 in \cite{indep-aec-apal}) $\s$ induces an independence relation $\is'$ for types of length at most $\mu$ over models of size $\mu$ that is well-behaved (i.e.\ it has all of the properties of a fully good independence relation except full model continuity and disjointness). By the canonicity of such relations (see the proofs of Corollary 5.19 and Theorem 6.13 in \cite{bgkv-apal}), $\is'$ must be the same as $\is^{\le \mu} \rest K_\mu$, the restriction of $\is$ to size $\mu$. Thus for all $\mu \ge \LS (K)$, $\is$ has extension for types of length at most $\mu$ over models of size $\mu$. By the proof of \cite[Lemma 14.13]{indep-aec-apal}, this suffices to conclude that $\is$ has extension.
\end{proof}
\begin{remark}
  The proof shows that instead of the AEC admitting intersections, it is enough to assume that for each $\mu$, the restriction of $\is$ to a good frame in $\mu$ has the existence property for uniqueness triples. Unfortunately the proof in \cite[Section 11]{indep-aec-apal} only works when the frame is restricted to the saturated models of size $\mu$.
\end{remark}

\begin{cor}\label{dap-cor}
  $K$ has disjoint amalgamation.
\end{cor}
\begin{proof}
  Because $\is$ has existence, extension and disjointness.
\end{proof}

Another consequence of having a closure operator is:

\begin{thm}[Finite character of independence]\label{fin-char-thm}
  $\nfs{M_0}{A}{B}{N}$ if and only if for all finite $A_0 \subseteq A$ and $B_0 \subseteq B$, $\nfs{M_0}{A_0}{B_0}{N}$. That is, $\is$ has the $(<\aleph_0)$-witness property.
\end{thm}
\begin{proof}
  By symmetry it is enough to show that if $\nfs{M_0}{A_0}{B}{N}$ for all finite $A_0 \subseteq A$, then $\nfs{M_0}{A}{B}{N}$. For each finite $A_0 \subseteq A$, let $M_{A_0} := \cl^N (|M_0| \cup A_0)$. Let $M := \cl^N (|M_0| \cup B)$.  By Lemma \ref{cl-indep}, $\nfs{M_0}{M_{A_0}}{M}{N}$ for each finite $A_0 \subseteq A$. Let $M_A := \cl^N (|M_0| \cup A)$. It is easy to see that $\seq{M_{A_0} \mid A_0 \in [A]^{<\aleph_0}}$ is a directed system with union $M_A$. Therefore by full model continuity, $\nfs{M_0}{M_A}{M}{N}$, and so $\nfs{M_0}{A}{B}{N}$.
\end{proof}

\begin{remark}
  One can check that $(K, \lea, \nf, \cl)$ satisfies the axiomatic framework $\Axfr$ from \cite[Chapter V.B]{shelahaecbook2}.
\end{remark}

For the next two results, we drop our hypotheses.

\begin{thm}\label{indep-thm}
  Let $K$ be a fully $(<\LS (K))$-tame and short AEC with amalgamation. Assume further that $K$ locally admits intersections. 

  If $K$ is categorical in a $\mu \ge \hanf{\LS (K)}$, then there exists $\lambda < \hanf{\LS (K)}$ such that $K_{\ge \lambda}$ is fully good. Moreover the independence relation has the $(<\aleph_0)$-witness property.
\end{thm}
\begin{proof}
  Combine Corollary \ref{fully-good-cor}, Theorem \ref{fully-good-ir}, and Theorem \ref{fin-char-thm}.
\end{proof}

\begin{remark}
  If $K$ is not categorical but only superstable (see Definition \ref{ss-def}), then we can generalize the result (using \cite[Theorem 15.1]{indep-aec-apal}) provided that for all $\lambda$, $\Ksatp{\lambda}$ (the class of $\lambda$-saturated models in $K$) locally admits intersections.
\end{remark}

\subsection{Set bases}

We end by showing that it is possible to extend the independence relation to define forking not only over models but also over sets. In the terminology of \cite{hyttinen-lessmann}, $K$ is simple (note  that the paper gives an example due to Shelah of a class that has a fully good independence relation, yet is not simple). 

For our arguments to work, we have to assume that $K$ admits intersections, i.e.\ not just locally. To see that this is not a big loss, recall that if $K$ is categorical in unboundedly many cardinals and has amalgamation, then the models in the categoricity cardinals are saturated, so for $M \in \LS (K)$, $K_M$ will also be categorical in unboundedly many cardinals.

\begin{hypothesis}
  $K$ admits intersections.
\end{hypothesis}

\begin{defin}
  Let $N \in K$ and $A, B, C \subseteq |N|$. Define $\nfs{A}{B}{C}{N}$ to hold if and only if $\nfs{\cl^N (A)}{\cl^N (AB)}{\cl^N (AC)}{N}$.

  We define properties such as invariance, monotonicity, etc.\ just as for the model-based version of independence.
\end{defin}

\begin{remark}
  When $A \lea N$, this agrees with the previous definition of independence.
\end{remark}

\begin{thm}\label{set-base-thm} \
  \begin{enumerate}
    \item $\nf$ has invariance, left and right monotonicity, base monotonicity, and normality.
    \item $\nf$ has symmetry, finite character (i.e.\ the $(<\aleph_0)$-witness property), existence and transitivity.
    \item $\nf$ has extension.
    \item Let $N \in K$ and let $\seq{B_i : i < \delta}$ be an increasing chain of sets. Let $B_\delta := \bigcup_{i < \delta} B_i$ and assume $B_\delta \subseteq |N|$. Let $p \in \gS^\alpha (B; N)$. If $\cf{\delta} > \alpha$, then there exists $i < \delta$ such that $p$ does not fork over $B_i$.
    \item If $p \in \gS^\alpha (B; N)$, there exists $A \subseteq B$ such that $p$ does not fork over $A$ and $|A| < |\alpha|^+ + \aleph_0$.
  \end{enumerate}
\end{thm}
\begin{proof} \
  \begin{enumerate}
    \item Easy. 
    \item Easy.
    \item By transitivity and extension of $\is$.
    \item By local character for $\is$.
    \item By finite character, it is enough to show it when $\alpha < \omega$. Work by induction on $\lambda := |B|$. If $\lambda < \aleph_0$, take $A = B$ and use the existence property. If $\lambda \ge \aleph_0$, write $B = \bigcup_{i < \lambda} B_i$, where $|B_i| < \lambda$ for all $i < \lambda$. By the previous result, there exists $i < \lambda$ such that $p$ does not fork over $B_i$. Now apply the induction hypothesis and transitivity.
  \end{enumerate}
\end{proof}
\begin{remark}
  Thus in this framework types of finite length really do not fork over a finite set. This removes the need for a special chain version of local character (i.e.\ if $\seq{M_i : i \le \delta}$ is increasing continuous, $p^{<\omega} \in \gS (M_\delta)$, there exists $i < \delta$ such that $p$ does not fork over $M_i$).
\end{remark}

\bibliographystyle{amsalpha}
\bibliography{aecs-intersection}

\end{document}